\documentclass[final,reqno]{amsart}

\usepackage{fullpage}
\usepackage[foot]{amsaddr}
\usepackage{mathtools,amssymb,mathabx,amsthm}
\usepackage{mathrsfs,euscript}
\usepackage{braket,slashed}
\usepackage[only llbracket,rrbracket,longmapsfrom]{stmaryrd}
\usepackage{tikz}
\usetikzlibrary{cd}
\usepackage{showkeys,comment}
\usepackage[l2tabu,orthodox]{nag}
\usepackage[all, warning]{onlyamsmath}

\usepackage[bookmarks=false,draft=false,breaklinks,colorlinks]{hyperref}

\usepackage{color}

\numberwithin{equation}{section}
\allowdisplaybreaks

{\par \vspace{\baselineskip}%
 \noindent \textbf{Acknowledgements.}}%
{\par \vspace{\baselineskip}}

\usepackage{enumitem}
\setenumerate{label=(\arabic*),nosep}
\setitemize{nosep}
\newlist{clist}{enumerate}{1}
\setlist*[clist]{label=(\roman*), nosep}

\usepackage{cleveref}

\crefname{thm}{Theorem}{Theorems}
\crefname{dfn}{Definition}{Definitions}
\crefname{prp}{Proposition}{Propositions}
\crefname{lem}{Lemma}{Lemmas}
\crefname{cor}{Corollary}{Corollaries}
\crefname{clm}{Claim}{Claims}
\crefname{fct}{Fact}{Facts}
\crefname{rmk}{Remark}{Remarks}
\crefname{eg}{Example}{Examples}
\crefname{figure}{Figure}{Figures}
\crefname{table}{Table}{Tables}
\crefname{section}{\S\!}{\S\S\!}
\crefname{subsection}{\S\!}{\S\S\!}
\crefname{subsubsection}{\S\!}{\S\S\!}
\crefname{appendix}{Appendix}{Appendices}
\crefname{equation}{}{}
\crefname{mainthm}{Theorem}{Theorems}

\theoremstyle{definition}
\newtheorem{thm}{Theorem}[section]
\newtheorem{dfn}[thm]{Definition}
\newtheorem{prp}[thm]{Proposition}
\newtheorem{lem}[thm]{Lemma}
\newtheorem{cor}[thm]{Corollary}

\newtheorem{fct}[thm]{Fact}
\newtheorem{rmk}[thm]{Remark}
\newtheorem{eg}[thm]{Example}
\newtheorem*{rmk*}{Remark}
\newtheorem{mainthm}{Theorem}

\AddToHook{env/dfn/begin}{\crefalias{thm}{dfn}}
\AddToHook{env/prp/begin}{\crefalias{thm}{prp}}
\AddToHook{env/lem/begin}{\crefalias{thm}{lem}}
\AddToHook{env/cor/begin}{\crefalias{thm}{cor}}
\AddToHook{env/clm/begin}{\crefalias{thm}{clm}}
\AddToHook{env/fct/begin}{\crefalias{thm}{fct}}
\AddToHook{env/rmk/begin}{\crefalias{thm}{rmk}}
\AddToHook{env/eg/begin}{\crefalias{thm}{eg}}

\newcommand{\bl}{\bullet}
\newcommand{\pd}{\partial}

\newcommand{\ve}{\varepsilon}

\newcommand{\ol}{\overline}
\newcommand{\ul}{\underline}
\newcommand{\wh}{\widehat}
\newcommand{\wt}{\widetilde}
\newcommand{\ceq}{\coloneqq} 
\newcommand{\sd}{\slashed{\pd}} 

\newcommand{\xr}{\xrightarrow}
\newcommand{\xrr}[1]{\xrightarrow{\ #1 \ }{}}

\newcommand{\srj}{\twoheadrightarrow}

\newcommand{\lto}{\longrightarrow}

\newcommand{\sto}{\xr{\sim}}

\newcommand{\mto}{\mapsto}
\newcommand{\lmto}{\longmapsto}

\newcommand{\cont}{\textup{cont}}
\newcommand{\loc}{\textup{loc}}
\newcommand{\rest}{\textup{rest}}

\newcommand{\bbA}{\mathbb{A}}
\newcommand{\bbC}{\mathbb{C}}

\newcommand{\bbZ}{\mathbb{Z}}

\newcommand{\clA}{\mathcal{A}}
\newcommand{\clC}{\mathcal{C}}
\newcommand{\clF}{\mathcal{F}}

\newcommand{\clJ}{\mathcal{J}}
\newcommand{\clL}{\mathcal{L}}

\newcommand{\clT}{\mathcal{T}}

\newcommand{\shO}{\EuScript{O}}

\newcommand{\frg}{\mathfrak{g}}
\newcommand{\frb}{\mathfrak{b}}
\newcommand{\frh}{\mathfrak{h}}
\newcommand{\frk}{\mathfrak{k}}
\newcommand{\frl}{\mathfrak{l}}

\newcommand{\frn}{\mathfrak{n}}
\newcommand{\frp}{\mathfrak{p}}

\newcommand{\frsl}{\mathfrak{sl}}
\newcommand{\frosp}{\mathfrak{osp}}
\newcommand{\frVir}{\mathfrak{Vir}}

\newcommand{\cSet}{\mathsf{Set}}

\newcommand{\SSch}{\mathsf{SSch}}

\newcommand{\abs}[1]{\left| #1 \right|}
\newcommand{\dbr}[1]{\llbracket #1 \rrbracket} 
\newcommand{\dpr}[1]{(\!( #1 )\!)}

\newcommand{\nop}[1]{\,\substack{\scriptscriptstyle\circ \\ \scriptscriptstyle\circ} #1 \substack{\scriptscriptstyle\circ \\ \scriptscriptstyle\circ\,}}

\DeclareMathOperator{\neu}{ne}
\DeclareMathOperator{\gr}{gr}

\DeclareMathOperator{\Coind}{Coind}

\DeclareMathOperator{\Ad}{Ad}

\DeclareMathOperator{\id}{id}

\DeclareMathOperator{\Vir}{Vir}
\DeclareMathOperator{\Aut}{Aut}
\DeclareMathOperator{\Der}{Der}
\DeclareMathOperator{\End}{End}

\DeclareMathOperator{\Hom}{Hom}
\DeclareMathOperator{\Img}{Im}

\DeclareMathOperator{\Lie}{Lie}

\DeclareMathOperator{\sres}{sres}
\DeclareMathOperator{\Spec}{Spec}

\DeclareMathOperator{\Sym}{Sym}
\DeclareMathOperator{\Vect}{Vect}
\DeclareMathOperator{\Fun}{Fun}
\DeclareMathOperator{\Wedge}{{\textstyle\bigwedge}}

\DeclareMathOperator{\shBer}{\EuScript{B}\kern-.1em\mathit{er}}
\DeclareMathOperator{\shDer}{\EuScript{D}\kern-.1em\mathit{er}}
\DeclareMathOperator{\shEnd}{\EuScript{E}\kern-.1em\mathit{nd}}
\DeclareMathOperator{\shHom}{\EuScript{H}\kern-.15em\mathit{om}}
\DeclareMathOperator{\shSym}{\EuScript{S}\kern-.1em\mathit{ym}}

\begin{document}
\title{Free field realizations of $N=1$ SUSY affine vertex algebras}
\author{Takumi Iwane}
\date{\today}
\address{Graduate School of Mathematics, Nagoya University.
 Furocho, Chikusaku, Nagoya, Japan, 464-8602.}
\email{takumi.iwane.c8@math.nagoya-u.ac.jp}

\begin{abstract}
    We construct Wakimoto-type free field realizations of $N=1$ SUSY affine vertex algebras associated with basic classical Lie superalgebras. Using the superfield formalism, we obtain a SUSY analogue of the Wakimoto realization in terms of $\beta\gamma$--$bc$ systems and boson--fermion systems (or SUSY Heisenberg vertex algebras). We also construct screening operators for the free field realizations.

    A distinctive feature of the SUSY setting is that, unlike the ordinary Wakimoto realization, the injectivity of the free field homomorphism depends on the level. We prove that the homomorphism is injective for nondegenerate $\kappa$, whereas at $\kappa=0$ it fails to be injective. Thus, the SUSY Wakimoto realization exhibits a degeneration at $\kappa=0$ that has no direct counterpart in the ordinary affine case.
\end{abstract}

\maketitle


\section{Introduction}

Free field realizations are a fundamental tool in the representation theory of affine Lie algebras and vertex algebras. The first example was constructed by Wakimoto for $\wh{\frsl}_2$ \cite{Wak86}, and B.\ Feigin and E.\ Frenkel subsequently extended the construction to arbitrary simple Lie algebras \cite{FF88,FF90}. These realizations embed affine vertex algebras into tensor products of Heisenberg vertex algebras and $\beta\gamma$ systems, producing the distinguished class of representations now known as Wakimoto modules. Their geometric and homological interpretations involve the big cell of the flag variety, semi-infinite induction, and chiral differential operators; see, for example, \cite{dBF97,Vor99,Fre05,Fre07,FG08}. At the critical level, they also play a central role in the Feigin--Frenkel description of the center of an affine vertex algebra in terms of opers \cite{Fre05,Fre07}. Free field realizations of affine Lie superalgebras have also been studied, for example, in \cite{IMP93,Ras98}.

The present paper concerns $N_K=1$ SUSY vertex algebras. Their generating currents are superfields, and their translation structure contains an odd operator $S$ satisfying $S^2=T$. Supersymmetric current algebras were first formulated in \cite{KT85}. Their formulation in the language of SUSY vertex algebras was developed by Heluani and Kac \cite{HK07}. SUSY affine vertex algebras are also the basic input for SUSY quantum Hamiltonian reduction and the theory of SUSY $W$-algebras \cite{MRS21,Song24,GSS25}. The purpose of this paper is to construct an intrinsic Wakimoto-type free field theory for these algebras in the superfield formalism.

We briefly recall SUSY affine vertex algebras, which are our main objects in this paper (for a precise definition, see \cref{eg:SUSY affine VA}). Let $\frg$ be a basic classical Lie superalgebra \cite{Kac77} with a supersymmetric invariant bilinear form $\kappa$ . Then, there exists a $N_K=1$ SUSY vertex algebra $V^\kappa_{N=1}(\frg)$ generated by superfields $\ol{x}(Z)$ ($x\in\frg$), whose parity is $\abs{x}+1$, with the OPEs given by
\[        
    [\ol{x}(Z),\ol{y}(W)]=(-1)^{\abs{x}}(\delta(Z,W)\ol{[x,y]}(W)+(\pd_W^{(0|1)}\delta(Z,W))\kappa(x,y)).
\]
When $\kappa$ is nondegenerate, there is also an indirect construction. After forgetting the SUSY structure, the component fields can be decoupled, giving an isomorphism of ordinary vertex superalgebras
\[
  V_{N=1}^{\kappa}(\frg)
  \simeq
  V^{\kappa+\kappa_c}(\frg)\otimes \clF_{\neu}^{\kappa}(\frg),
\]
where $\kappa_c$ is the bilinear form corresponding to the critical level and $\clF_{\neu}^{\kappa}(\frg)$ is the neutral free-fermion vertex superalgebra associated with $\frg$. One may then apply the ordinary Wakimoto homomorphism to the first tensor factor and leave the fermionic factor unchanged. This procedure produces a free field homomorphism at the level of component fields, but it suppresses the underlying SUSY geometry. Moreover, the decoupling transformation is singular at $\kappa=0$.

Our construction is carried out directly in the superfield formalism. We begin with the right action of $\frg$ on the big cell of the flag supervariety and the base affine superspace, and pass to the induced action of the SUSY loop algebra $\clL^{N=1}\frg$ on the superloop space. Following Frenkel’s construction in the ordinary setting \cite{Fre05}, we then lift this action to the free field system by normal ordering. The resulting normal-ordering anomaly is removed by comparing the corresponding local Chevalley–Eilenberg cocycle with the affine cocycle. This yields the first main result.

\begin{mainthm}[\cref{thm:free field realizations}]\label{thmA}
Let $\frg$ be a basic classical Lie superalgebra and let $\kappa$ be an invariant supersymmetric bilinear form on $\frg$. There exists a homomorphism of SUSY vertex algebras
\[
  w_\kappa\colon
  V_{N=1}^{\kappa}(\frg)
  \lto
  M_{\Delta_+}\otimes\Pi^\kappa,
\]
where $M_{\Delta_+}$ is the $\beta\gamma$--$bc$ system associated with the big cell of the flag supervariety and $\Pi^\kappa$ is the Cartan boson--fermion system. The homomorphism is given by explicit superfield formulas for the Chevalley generators of $\frg$.
\end{mainthm}
Unlike the ordinary Wakimoto formulas, the resulting Cartan system carries the original form $\kappa$ without a level shift. This reflects the cancellation between the $\beta\gamma$ and $bc$ contributions to the normal-ordering anomaly, analogous to the cancellation underlying the chiral de Rham complex \cite{MSV99,BZHS08}. If $\frg$ is purely even, then this construction appeared in \cite{ATY91} to solve the SUSY Knizhnik--Zamolodochikov equation \cite{DKPR84}.

The second main result is that injectivity depends on the level, a phenomenon with no direct counterpart in the ordinary Wakimoto realization.

\begin{mainthm}[\cref{thm:injective,thm:collapse injectivity}]\label{thmB}
If $\kappa$ is nondegenerate, then $w_\kappa$ is injective. If $\kappa=0$, then $w_\kappa$ is not injective.
\end{mainthm}

For nondegenerate $\kappa$, we identify the associated graded homomorphism with the pullback along the arc-space morphism $\clJ\mu$ induced by
\[
  \mu\colon
  T^*U\times \frh^*\times
  \Pi(T^*U\times\frh^*)
  \lto
  \frg^*\times\Pi\frg^*.
\]
In contrast to the ordinary affine case, the morphism $\mu$ depends on $\kappa$ (\cref{prp:C_2-algebra}). If $\kappa$ is nondegenerate, then the morphism $\mu$ is smooth on a nonempty open subset, which implies the injectivity of the induced pullback on arc spaces. If $\kappa=0$, we have an explicit nonzero element in the kernel of $w_0$: a suitable specialization of the Kac--Todorov superconformal vector yields a nonzero element in the kernel of $w_0$. Thus the superfield construction remains defined at level zero but undergoes a genuine degeneration there. This behavior cannot be detected by the component-field construction, whose decoupling transformation is singular at $\kappa=0$.

We next construct screening operators associated with the simple roots. To this end, we first extend the free field realization to standard parabolic subsuperalgebras and establish its factorization through rank-one Levi subsuperalgebras. For a standard parabolic $\frp_I=\frl_I\oplus\frn_{I,-}$, we construct a homomorphism
\[
  w_\kappa^{\frp_I}\colon
  V_{N=1}^{\kappa}(\frg)
  \lto
  M_{\Delta(\frn_{I,+})}
  \otimes V_{N=1}^{\kappa}(\frl_I).
\]
When $I$ consists of a single simple root $\alpha_i$, the original homomorphism $w_\kappa$ factors through the SUSY affine vertex algebra of the corresponding rank-one Levi subsuperalgebra. Up to the remaining abelian Cartan summands, this Levi subsuperalgebra is $\frsl_2$, $\frosp_{1|2}$, or $\frsl_{1|1}$ according as $\alpha_i$ is even, odd non-isotropic, or odd isotropic. These rank-one factorizations lead uniformly to screening supercurrents for all three types of simple roots.

\begin{mainthm}[\cref{thm:screening operator,cor:screening operator}]\label{thmC}
Assume that $\kappa$ is nondegenerate. For each simple root $\alpha_i$, the corresponding screening operator $S_i$ commutes with the action of the $N=1$ SUSY affine Lie superalgebra. Consequently, identifying $V^\kappa_{N=1}(\frg)$ with the image of $w_\kappa$, we have 
\[
  V_{N=1}^{\kappa}(\frg)
  \subset
  \bigcap_{i=1}^{\dim\frh}\ker S_i.
\]
\end{mainthm}

We conclude the introduction by mentioning several directions for future work. First, it would be interesting to understand more precisely the degeneracy phenomenon appearing in \cref{thmB}, and in particular to clarify its representation-theoretic and geometric meaning. Second, for \cref{thmC}, a natural problem is to establish the reverse inclusion at generic level, thereby obtaining an exact description of the image in terms of screening operators. Beyond these questions, it would be interesting to clarify the relation between the free field realization constructed in this paper and the free field realizations of SUSY W-algebras \cite{Song24}. Finally, it would be particularly interesting to investigate the connection with super opers \cite{Zei15,Zei25}. In the ordinary setting, free field realizations, screening operators, W-algebras, and opers are different aspects of the same structure. It would be interesting to clarify to what extent such a picture persists in the supersymmetric setting.

\subsection*{Organization.}
In Section~2, we recall $N_K=1$ SUSY vertex algebras, SUSY affine Lie superalgebras, boson--fermion systems, and $\beta\gamma$--$bc$ systems. In Section~3, we construct the vector-field realization on the big cell, pass to the superloop space, and formulate the normal-ordering obstruction. In Section~4, we introduce the local Chevalley--Eilenberg complex and compare the normal-ordering cocycle with the affine cocycle by restriction to the Cartan subalgebra. In Section~5, we construct the free field homomorphism, determine the image of the Kac--Todorov superconformal vector, prove injectivity for nondegenerate $\kappa$, and exhibit the failure of injectivity at $\kappa=0$. In Section~6, we construct parabolic free field homomorphisms and their factorizations through rank-one Levi subsuperalgebras. In Section~7, we construct screening supercurrents and prove that the corresponding screening operators commute with the SUSY affine action. Appendix~A develops the completed mode algebras of SUSY vertex algebras used in the construction.


\subsection*{Global notation and terminology}

We summarize here the notation and terminology used throughout the main text.
\begin{itemize}
\item
The symbols $\bbZ$ and $\bbC$ denote 
the set of all integers and complex numbers, respectively.

\item
Unless otherwise specified, vector superspaces are defined over the field $\bbC$, 
and $\Hom$, $\End$, and $\otimes$ are understood in this category.

\item
For a vector superspace $V$, we denote the linear dual $\Hom(V,\bbC)$ by $V^*$.

\item 
For a homogeneous element $v$ of a vector superspace $V$, we denote its parity (super degree) by $\abs{v}\in\bbZ/2\bbZ$.

\item 
For a vector superspace $V$, $\Pi V$ denotes its parity change, i.e.\ $(\Pi V)_0\ceq V_1$ and $(\Pi V)_1\ceq V_0$.

\item 
Unless otherwise specified, supercommutative algebra and superscheme are defined over the field $\bbC$. 

\item
For a superscheme $X$, $\Fun(X)$ and $\Vect(X)$ denote the supercommutative algebra of functions on $X$ and the Lie superalgebra of vector fields on $X$, respectively.

\item
We will freely identify a superscheme $X$ with its functor of points $X\colon\SSch\to\cSet$ without further comment, and we also write $X(S)\ceq X(\Spec(S))$ for any supercommutative algebra $S$ by abuse of notation.

\end{itemize}


\section{Preliminaries: SUSY vertex algebras} 

    In this paper, we work exclusively with $N_K=1$ SUSY vertex algebras in the sense of \cite{HK07}. Readers interested in the general theory of SUSY vertex algebras are referred to \cite{HK07}. 

\subsection{SUSY vertex algebras}
    Let $Z=(z,\theta)$ and $W=(w,\zeta)$ be supervariables, where $z, w$ are even and $\theta, \zeta$ are odd. We define a new supervariable $Z-W$ by 
    \[
        Z-W\ceq (z-w-\theta\zeta,\theta-\zeta).
    \]

    For a vector superspace $V$, we define the vector superspaces $V\dbr{Z}$ and $V\dpr{Z}$ by
    \[
        V\dbr{Z}\ceq\{\sum_{j\geq 0;J=0,1}Z^{j|J}a_{j|J}\mid a_{j|J}\in V\},
        \quad
        V\dpr{Z}\ceq\{\sum_{j>\infty;J=0,1}Z^{j|J}a_{j|J}\mid a_{j|J}\in V\},
    \]
    where $Z^{j|J}$ denotes $z^j\theta^J$. 

    \begin{dfn}[\cite{HK07}]
    A \emph{supersymmetric vertex algebra} (SUSY vertex algebra for short) $V=(V,\ket{0},T,S)$ consists of the following data:
    \begin{itemize}
        \item
        (space of states)
        a vector superspace $V$,
        \item 
        (vacuum vector)
        a vector $\ket{0}\in V$ with parity $0$, 
        \item 
        (translation operator)
        an operator $S$ on $V$ with parity $1$ and an operator $T$ on $V$ with parity $0$,
        \item 
        (state-field correspondence)
        a parity preserving map 
        \[
            Y(-,Z)\colon V\to\Hom(V,V\dpr{Z});\quad Y(a,Z)=\sum_{j\in\bbZ;J=0,1}Z^{-j-1|1-J}a_{(j|J)},
        \] 
    \end{itemize}
    satisfying the following axioms:
    \begin{enumerate}
        \item 
        (vacuum axiom)
        $T\ket{0}=S\ket{0}=0$, $Y(\ket{0},Z)=\id_V$. For any $a\in V$, $Y(a,Z)\ket{0}$ is regular at $Z=0$ and $Y(a,Z)\ket{0}|_{z=0,\theta=0}=a$.
        \item
        (translation axiom)
        $[S,S]=2T$, and $[S,Y(a,Z)]=(\pd_\theta-\theta\pd_z)Y(a,Z)$, $[T, Y(a,Z)]=\pd_zY(a,Z)$.
        \item 
        (locality axiom)
        For any vectors $a,b\in V$, superfields $Y(a,Z),Y(b,W)$ are local, i.e.\ there exists $n\gg 0$ such that $(z-w)^n[Y(a,Z),Y(b,W)]=0$.
    \end{enumerate}
\end{dfn}

\begin{dfn}
    Let $a(Z)$, $b(Z)$ be superfields, i.e.\ $a(Z), b(Z)\in\Hom(V,V\dpr{Z})$. We define $\nop{a(Z)b(W)}$ by
    \[
        \nop{a(Z)b(W)}\ceq a(Z)_+b(W)+(-1)^{\abs{a}\abs{b}}b(W)a(Z)_-,
    \]
    where $a(Z)_+\ceq \sum_{j\geq 0, J=0,1}Z^{-j-1|1-J}a_{(j|J)}$, $a(Z)_-\ceq \sum_{j< 0, J=0,1}Z^{-j-1|1-J}a_{(j|J)}$. Then, one can show that $\nop{a(Z)b(Z)}$ is a well-defined superfield. We call $\nop{a(Z)b(W)}$, $\nop{a(Z)b(Z)}$ the \emph{normally ordered product} of $a(Z)$ and $b(W)$.
\end{dfn}

We define a formal delta superfunction $\delta(Z,W)$ by
\[
    \delta(Z,W)\ceq (\theta-\zeta)\delta(z,w),
\]
where $\delta(z,w)$ is a formal delta function, i.e.\ $\delta(z,w)=\sum_{n\in\bbZ}z^nw^{-n-1}$. We set $\pd_Z\ceq(\pd_z,\sd_z)$, $\sd_z\ceq \pd_\theta+\theta\pd_z$, and 
\[
    \pd_Z^{(j|J)}\ceq \frac{(-1)^{J(1+J)/2}}{j!}\pd_z^j\sd_z^J
    =
    \left\{
    \begin{array}{cc}
        -\dfrac{\pd_z^j\sd_z}{j!} & (J=1),\\
        \dfrac{\pd_z^j}{j!} & (J=0)
    \end{array}
    \right.
\]
for $j\geq 0$, $J=0,1$. Note that the following identity holds:
\[
    \pd_Z^{(j|J)}\delta(Z,W)=(Z-W)^{-j-1|1-J}-(-W+Z)^{-j-1|1-J},
\]
where the first (resp.\ second) term of the right-hand side means expansion in $\abs{z}>\abs{w}$ (resp.\ $\abs{w}>\abs{z}$). We recall the basic facts about SUSY vertex algebras:
\begin{prp}[\cite{HK07}]
    In a SUSY vertex algebra $V$, the following statements hold:
    \begin{enumerate}
        \item 
        $Y(a_{(-1|1)}b,Z)=\nop{Y(a,Z)Y(b,Z)}$,
        \item
        $Y(Ta,Z)=\pd_zY(a,Z)$, $Y(Sa,Z)=\sd_zY(a,Z)$,
        \item 
        The operator product expansion (OPE for short):
        \[
            [Y(a,Z),Y(b,W)]=\sum_{j\geq 0;J=0,1}(\pd_W^{(j|J)}\delta(Z,W))Y(a_{(j|J)}b,W),
        \]
    \end{enumerate}
\end{prp}

\begin{rmk}
    We also write the OPE as follows:
    \[
        Y(a,Z)Y(b,W)\sim \sum_{j\geq0; J=0,1}(Z-W)^{-j-1|1-J}Y(a_{(j|J)}b,W),
    \]
    where $\sim$ denotes equality modulo terms regular in $z-w$.
\end{rmk}
    
\subsection{Examples of SUSY vertex algebras}
    We collect some examples of SUSY vertex algebras.

    Let $\frg$ be a Lie superalgebra with an invariant supersymmetric form $\kappa$, which is not necessarily nondegenerate. We consider the Lie superalgebra given by $\clL^{N=1}\frg\ceq\frg\dpr{T}$, where $T$ is a supervariable $T=(t,\omega)$. Then, we define a linear map $\sigma_\kappa\colon\clL^{N=1}\frg\otimes \clL^{N=1}\frg\to\bbC$ by $\sigma_\kappa(fx,gy)\ceq (-1)^{\abs{x}\abs{g}}\oint dT(\sd_tf)g\kappa(x,y)$ for $f,g\in\bbC\dpr{T}$, and $x,y\in\frg$, which is an even $2$-cocycle of $\clL^{N=1}\frg$. Here, $\oint dT$ denotes the superresidue with respect to $T$, i.e.\ the operation of taking the coefficient of $\omega t^{-1}$. 
    
    \begin{dfn}\label{dfn:SUSY affine Lie superalgebra}
        For a Lie superalgebra $\frg$ with an invariant supersymmetric form $\kappa$, we define a Lie superalgebra $\wh{\frg}^\kappa_{N=1}$ by $\wh{\frg}^\kappa_{N=1}\ceq\clL^{N=1}\frg\oplus\bbC K$ with a Lie bracket given by
        \[
            [fx,gy]\ceq (-1)^{\abs{x}\abs{g}}fg[x,y]+\sigma_\kappa(fx,gy)=(-1)^{\abs{x}\abs{g}}(fg[x,y]+\oint dT(\sd_tf)g\kappa(x,y)K),
        \]
        and $K$ is central. We call $\wh{\frg}^\kappa_{N=1}$ the \emph{SUSY affine Lie superalgebra} of $\frg$.

        If we set $x_n\ceq t^nx$, $\ol{x}_{n+1/2}\ceq \omega t^nx$ for $x\in\frg$, $n\in\bbZ$, then we can write the above relation in $\wh{\frg}^\kappa_{N=1}$ explicitly:
        \[
            [x_n,y_m]\ceq [x,y]_{n+m}+n\delta_{n+m,0}\kappa(x,y)K,
            \quad
            [\ol{x}_n,y_m]\ceq \ol{[x,y]}_{n+m},
            \quad
            [\ol{x}_n,\ol{y}_m]\ceq (-1)^x\delta_{n+m,0}\kappa(x,y)K,
        \]
        and $K$ is central. If we define the formal superseries $\ol{x}(Z)$ with coefficients in $\wh{\frg}^\kappa_{N=1}$ by $\ol{x}(Z)=\sum_n(z^{-n-1}\ol{x}_{n+1/2}+\theta z^{-n-1}x_n)$, then the above relation is equivalent to
        \[
            [\ol{x}(Z),\ol{y}(W)]=(-1)^{\abs{x}}(\delta(Z,W)\ol{[x,y]}(W)+(\pd_W^{(0|1)}\delta(Z,W))\kappa(x,y)K).
        \]
        for $x,y\in\frg$.
    \end{dfn}
    
    \begin{eg}[SUSY affine vertex algebra]\label{eg:SUSY affine VA}
        Let $\bbC\ket{0}$ be a one-dimensional $\frg\dbr{T}\oplus\bbC K$-module defined by $\frg\dbr{T}\ket{0}=0$, $K\ket{0}=\ket{0}$, and let $V^\kappa_{N=1}(\frg)$ be the induced $\wh{\frg}^\kappa_{N=1}$-module, i.e.\ $V^\kappa_{N=1}(\frg)=U(\wh{\frg}^\kappa_{N=1})\otimes_{U(\frg\dbr{T}\oplus\bbC K)}\bbC\ket{0}$. $V^\kappa_{N=1}(\frg)$ carries a unique SUSY vertex algebra structure such that $\ket{0}$ is the vacuum vector, and for $x\in\frg$, the superfield $Y(\ol{x}_{-1/2}\ket{0},Z)$ is equal to $\ol{x}(Z)$ on $V^\kappa_{N=1}(\frg)$, defined in \cref{dfn:SUSY affine Lie superalgebra}. The OPEs are
        \[
            [\ol{x}(Z),\ol{y}(W)]=(-1)^{\abs{x}}(\delta(Z,W)\ol{[x,y]}(W)+(\pd_W^{(0|1)}\delta(Z,W))\kappa(x,y))
        \]
        for $x,y\in\frg$.
    \end{eg}

    \begin{eg}[boson-fermion system]
        In the case of an abelian Lie superalgebra $\frg=\frh$, we call $\Pi^\kappa=\Pi^\kappa_0\ceq V^\kappa_{N=1}(\frh)$ the \emph{boson-fermion system} (or \emph{SUSY Heisenberg vertex algebra}). Note that $\Pi^\kappa$ is simply the tensor product of $\pi^\kappa$ and $\clF^\kappa_{\neu}$ as a vertex algebra. Here, $\pi^\kappa$ is a vertex algebra generated by $b_i(Z)$, whose parity is $\abs{b_i}=\abs{h_i}$, with an OPE
        \[
            [b_i(z),b_j(w)]=\pd_w\delta(z,w)\kappa(h_i,h_j),
        \]
        and $\clF^\kappa_{\neu}$ is a vertex algebra generated by $\ol{b_i}(z)$, whose parity is $\abs{\ol{b_i}}=\abs{h_i}+1$, with an OPE
        \[
            [\ol{b}_i(z),\ol{b}_j(w)]=(-1)^{\abs{h_i}}\delta(z,w)\kappa(h_i,h_j),
        \]
        where $\{h_i\}$ is a basis of $\frh$.
    \end{eg}

\begin{dfn}\label{dfn:SUSY Weyl algebra}
    Let $I$ be a finite set with parity. Let $\clA^I$ be an algebra generated by $a_{i,n},a^{*,i}_n$ ($i\in I,n\in\bbZ$) and $\ol{a}_{i,n},\ol{a}^{*,i}_n$ ($i\in I,n\in\bbZ+1/2$), whose parities are $\abs{a_{i,n}}=\abs{a^{*,i}_n}=\abs{i}$, $\abs{\ol{a}_{i,n}}=\abs{\ol{a}^{*,i}_n}=\abs{i}+1$, with relations given by
    \[
        [a_{i,n},a^{*,j}_m]\ceq \delta^j_i\delta_{n+m,0},
        \quad
        [\ol{a}^{*,i}_n,\ol{a}_{j,m}]\ceq \delta^i_j\delta_{n+m,0}.
    \]
    We call $\clA^I$ the \emph{SUSY Weyl algebra} associated with $I$.
\end{dfn}

\begin{eg}[$\beta\gamma$--$bc$ system]\label{eg:bc-betagamma system}
    Let $\clA^I_{+}$ be a subalgebra of $\clA^I$ generated by $a^{*,i}_{n}$, $\ol{a}^{*,i}_{n-1/2}$, $a_{i,n-1}$, $\ol{a}_{i,n-1/2}$ for $i\in I$, $n\geq 1$ and let $\bbC\ket{0}$ be the trivial $\clA^I_+$-module. Then, the induced $\clA^I$-module $M_I$ carries a unique SUSY vertex algebra structure such that $\ket{0}$ is a vacuum vector, and  
    \[
        \ol{a}_i(Z)\ceq Y(\ol{a}_{i,-1/2},Z)=\sum_nz^{-n-1}(\ol{a}_{i,n+1/2}+\theta a_{i,n}),
        \quad
        a^{*,i}(Z)\ceq Y(\ol{a}^{*,i}_0,Z)=\sum_nz^{-n}(a^{*,i}_n+\theta \ol{a}^{*,i}_{n-1/2}).
    \]
    for $i\in I$. The OPEs are
    \[
       [\ol{a}_i(Z),a^{*,j}(W)]
       =[a^{*,j}(Z),\ol{a}_i(W)]
       =\delta^j_i\delta(Z,W),
    \]
    and the others are trivial. 
\end{eg}

\begin{dfn}\label{dfn:N=1 super Virasoro algebra}
    Let $\frVir^c_{N=1}$ be the Lie superalgebra generated by $L_n$, $G_{n+1/2}$ ($n\in\bbZ$), $C$, whose parities are $\abs{L_n}=\abs{C}=0$, $\abs{G_{n+1/2}}=1$, with Lie bracket given by  
    \begin{align*}
        [L_m,L_n]
        &=(m-n)L_{m+n}
        +\frac{c}{12}(m^3-m)\delta_{m+n,0}C,\\
        [L_m,G_r]
        &=\left(\frac{m}{2}-r\right)G_{m+r},\\
        [G_r,G_s]
        &=2L_{r+s}
        +\frac{c}{3}\left(r^2-\frac{1}{4}\right)\delta_{r+s,0}C
    \end{align*}
    for $n,m\in\bbZ$ and $r,s\in\bbZ+1/2$, and $C$ is central. We call it the \emph{$N=1$ super Virasoro Lie algebra} (or \emph{Neveu--Schwarz algebra}) of central charge $c\in\bbC$.
\end{dfn}

\begin{eg}[$N=1$ super Virasoro vertex algebra]
    Let $\frVir^c_{N=1,+}$ be a subalgebra generated by $L_{n}$, $G_{n+1/2}$ for $n>-2$. Let $\bbC\ket{0}$ be a $\frVir^c_{N=1,+}\oplus\bbC C$-module defined by $L_n\ket{0}=G_{n+1/2}\ket{0}=0$, $C\ket{0}=\ket{0}$, and $\Vir^c_{N=1}$ be the induced $\frVir^c_{N=1}$-module. Then, $\Vir^c_{N=1}$ carries a unique SUSY vertex algebra structure such that $\ket{0}$ is a vacuum vector, and 
    \[
        G(Z)\ceq Y(G_{-3/2}\ket{0},Z)=\sum_n(z^{-n-2}G_{n+1/2}+2\theta z^{-n-2}L_n).
    \]
    The OPE of $G(Z)$ is
    \[
        [G(Z),G(W)]=2\delta(Z,W)\pd_W^{(1|0)}G(W)+(\pd_W^{(0|1)}\delta(Z,W))\pd_W^{(0|1)}G(W)+3\pd_W^{(1|0)}\delta(Z,W)G(W)+\frac{c}{3}\pd_W^{(2|1)}\delta(Z,W).
    \]
\end{eg}

\section{Vector field realization}

In this section, we construct the vector-field realization underlying the free field homomorphism. We first describe the realization on the big cell and then pass to the $N=1$ superloop space and the corresponding completed SUSY Weyl algebra.

\subsection{Flag supervariety}\label{ss:Flag supervariety}

In this subsection, we seek a vector field realization on the big cell of the flag supervariety and the base affine superspace. For  homogeneous superspaces, we refer to \cite{MZ11,FKT21}; for basic classical Lie superalgebras, we refer to \cite{Kac77,Mus12}. 

Let $\frg$ be a basic classical Lie superalgebra of rank $l$. We fix a Cartan subalgebra $\frh$. $\Delta\subset\frh^*$ denotes a root system. We fix a positive part $\Delta^+$, and let $\alpha_1, \ldots, \alpha_l\in\Delta^+$ be the corresponding simple roots. Let $\Delta^-$ be the set of negative roots, i.e.\  $\Delta^-=-\Delta^+$. The corresponding triangular decomposition is written as
\[
    \frg=\frn_+\oplus\frh\oplus\frn_-,\quad \frn_{\pm}\ceq \bigoplus_{\alpha\in\Delta^{\pm}}\frg_\alpha,
\]
where $\frg_\alpha$ is the root space of $\alpha$. We set the corresponding BPS (Borel--Penkov--Serganova) subalgebras $\frb_\pm\ceq \frh\oplus\frn_\pm$.

Let $G$ be a connected algebraic supergroup corresponding to $\frg$, and $B_\pm$ (resp.\ $N_\pm$, $H$) be subgroups of $G$ corresponding to $\frb_\pm$ (resp.\ $\frn_\pm$, $\frh$). Consider the base affine superspace $N_-\backslash G$ and the flag supervariety $B_-\backslash G$. Note that the natural map $N_-\backslash G\to B_-\backslash G$ is a principal $H$-bundle. Let $U$ be the big cell, i.e.\ $U\ceq B_-\backslash B_-N_+\simeq N_+$. Since $N_+$ is a unipotent group, the exponential map $\frn_+\sto N_+$ is an isomorphism of superschemes (\cite{MO05}; see also \cite[Corollary~1.6]{Nak20}). Therefore, we have an isomorphism  $\bbA^{\Delta_+}\simeq U$ as superschemes. For a positive root $\alpha\in\Delta_+$, $y^\alpha$ denotes the element  of $\frn_+^*\subset\Fun(U)$ corresponding to the dual basis of $e_\alpha$. Then, $f\in\Fun(U)=\bbC[y^\alpha]_{\alpha\in\Delta^+}$ is regarded  as $f(y^\alpha)=f(\exp(\sum_\alpha y^\alpha e_\alpha))$. The homomorphism $\rho=\rho_{\frb_-}\colon\frg\to\Vect(U)$ induced by the right action is written as follows for the Chevalley basis $\{e_\alpha, h_i, f_\alpha\}$:
    \begin{equation}\label{action of Lie algebra}
        \rho(e_\alpha)=\pd_{\alpha}+\sum_{\beta\in\Delta_+\cap(\Delta_++\alpha)}P^\beta_\alpha(y)\pd_\beta,
        \quad
        \rho(h_i)=-\sum_{\beta\in\Delta_+}\beta(h_i)y^\beta\pd_\beta,
        \quad
        \rho(f_\alpha)=\sum_{\beta\in\Delta_+}Q^\beta_\alpha(y)\pd_\beta,
    \end{equation}
    for some polynomial $P^\beta_\alpha(y), Q^\beta_\alpha(y)\in\bbC[y^\alpha]_{\alpha\in\Delta^+}$. In what follows, if $\alpha=\alpha_i$ is a simple root, we write $P^\beta_i(y)\ceq P^\beta_{\alpha_i}$, $Q^\beta_i(y)\ceq Q^\beta_{\alpha_i}$.

Since $N_-\backslash G\to B_-\backslash G$ is a trivial $H$-bundle on $U$, the space of left $H$-invariant vector fields on $(N_-\backslash G)|_U$ is naturally isomorphic to $\Vect(U)\oplus \Fun U\otimes \frh$.  

\begin{prp}
    There exists a homomorphism $\rho_{\frn_-}\colon\frg\to\Vect(U)\oplus \Fun U\otimes\frh$ such that
    \begin{align*}
        \rho_{\frn_-}(e_\alpha)&=\pd_{\alpha}+\sum_{\beta\in\Delta^+\cap(\Delta^++\alpha)}P^\beta_\alpha(y)\pd_\beta,\\
        \rho_{\frn_-}(h_i)&=-\sum_{\beta\in\Delta_+}\beta(h_i)y^\beta\pd_\beta+u_i,\\
        \rho_{\frn_-}(f_\alpha)&=\sum_{\beta\in\Delta^+}Q^\beta_\alpha(y)\pd_\beta+\sum_{i=1}^l R^i_\alpha(y)u_i,
    \end{align*}
    where $u_i\ceq h_i$, and the polynomials $P^\beta_\alpha(y),Q^\beta_\alpha(y)$ are those in \cref{action of Lie algebra}. Furthermore, if $\alpha=\alpha_i$ is a simple root, 
    \[
        \rho_{\frn_-}(f_i)=\sum_{\beta\in\Delta^+}Q^\beta_i(y)\pd_\beta+y^{\alpha_i}u_i.
    \]
\end{prp}

On the other hand, identifying $U$ with $N_+$, we can consider the left $N_+$-action on $U$. We write the induced homomorphism $\rho^L\colon\frn_+\to\Vect(U)$ as:
    \begin{equation}\label{left action of Lie algebra}
        \rho^L(e_\alpha)=-\pd_{\alpha}+\sum_{\beta\in\Delta^+\cap(\Delta^++\alpha)}P^{L,\beta}_\alpha(y)\pd_\beta
    \end{equation}
for some polynomials $P^{L,\beta}_\alpha(y)\in\bbC[y^\alpha]_{\alpha\in\Delta^+}$. In what follows, if $\alpha=\alpha_i$ is a simple root, we write $P^{L,\beta}_i(y)\ceq P^{L,\beta}_{\alpha_i}$. Note that $\rho(\frn_+)$ and $\rho^L(\frn_+)$ commute by construction.  

\subsection{The superloop space}
\begin{dfn}
    For a superscheme $X$, there exists a unique set-valued sheaf $\clL^{N=1} X$ on $\SSch$ such that 
    \[
        \clL^{N=1}X(S)=X(S\dpr{T})
    \]
    for a supercommutative algebra $S$, where $T$ is a supervariable $T=(t,\omega)$. We call $\clL^{N=1}X$ the \emph{superloop space} on $X$.
\end{dfn}

\begin{rmk}
    If $X$ is an affine superscheme, then its superloop space $\clL^{N=1}X$ is represented by an ind-affine superscheme $\{\clL^{N=1}_MX\}_{M\leq 0}$. In particular, if $X$ is an affine superspace $X=\bbA^{n|m}$, $\Fun X=\bbC[y^i]$, then $\clL^{N=1}_MX=\Spec\bbC[y^i_n|\psi^i_{n+1/2}]_{n\geq M}$, where $y^i_n$ (resp.\ $\psi^i_{n+1/2}$) is an indeterminate with parity $\abs{y^i}$ (resp.\ $\abs{y^i}+1$).
\end{rmk}

\begin{lem}\label{lem:superloop group and superloop Lie superalgebra}
    Let $G$ be an affine supergroup. Then, $\clL^{N=1}G$ is an ind-affine supergroup. Its Lie superalgebra $\Lie(\clL^{N=1}G)=\Der(\Fun(G),\bbC\dpr{T})$ is isomorphic to $\clL^{N=1}\frg=\frg\dpr{T})$. More precisely, for $x\in\frg$ and $j\in\bbZ$, $J=0,1$, the element $T^{j|J}x\in\clL^{N=1}\frg$ is identified with a continuous derivation $T^{j|J}x\colon\Fun(G)\to\bbC\dpr{T}$ given by 
    \[
        (T^{j|J}x)(a)\ceq (-1)^JT^{j|J}x(a)
    \]
    for $a\in\Fun(G)$.
\end{lem}

\begin{rmk}
    There is a more natural identification, which sends $T^{j|J}x\in\clL^{N=1}\frg$ to a continuous derivation $a\lmto T^{j|J}x(a)$. However, this identification is not well suited to the superfield formalism.
\end{rmk}

Let $G$ be an affine supergroup acting on a superscheme $X$ on the right. For any open subsuperscheme $U$, the action induces a homomorphism of Lie superalgebras $\rho\colon\frg\to\Vect(U)$. Then, we have a homomorphism of Lie superalgebras $\clL^{N=1}\rho\colon\clL^{N=1}\frg\to\Vect(\clL^{N=1}U)$. When $U$ is an affine superspace $\bbA^{n|m}$, we have an explicit description of $\clL^{N=1}\rho$.

\begin{prp}\label{prp:vect field realization on superloop space}
    In the above setting, if $\rho(x)$ is written as $\rho(x)=\sum_{i}x^i(y)\pd_i$, where $x^i(y)\in\Fun(U)$, $\pd_i=\frac{\pd}{\pd y^i}$, for $x\in\frg$, then the homomorphism $\clL^{N=1}\rho$ is 
    \[
        \clL^{N=1}\rho(\ol{x}(Z))=\sum_{i\in I}(-1)^{\abs{x^i}}x^i(y(Z))\ol{\pd}_i(Z).
    \]
    Here, $\ol{x}(Z)$ is given in \cref{dfn:SUSY affine Lie superalgebra}, and we use the following formal superseries: 
    \[
        y^i(Z)=\sum_j(z^jy^i_j+\theta z^j\psi^i_{j+1/2}),
        \quad  
        \ol{\pd}_i(Z)=\sum_j\left(z^{-j-1}(-1)^i\frac{\pd}{\pd \psi^i_{j+1/2}}+\theta z^{-j-1}\frac{\pd}{\pd y^i_j}\right).
    \]
\end{prp}

\begin{proof}
    By \cref{lem:superloop group and superloop Lie superalgebra}, we have
    \[
        \clL^{N=1}\rho(\ol{x}(Z))y^i(T)
        =\sum_{j;J}Z^{-j-1|1-J}(-1)^JT^{j|J}x^i(y(T))
        =(-1)^{\abs{x^i}}x^i(y(Z))\delta(Z,T).
    \]
    On the other hand, since the following holds 
    \[
        \ol{\pd}_i(Z)y^j(T)=\delta_i^j\delta(Z,T),
    \]
    we have the claim.
\end{proof}

\subsection{The completed SUSY Weyl algebra}
Let $\clA\ceq\clA^{\Delta^+}$ be a SUSY Weyl superalgebra of $\Delta^+$. Then, $\clA$ (and its naive completion $\clA^\natural$, see \cref{ss:geom interpretation}) acts on $\bbC[y^\alpha_n|\eta^\alpha_n]_{n\in\bbZ}$ as follows:
\[
    a_{\alpha,n}=\frac{\pd}{\pd y^\alpha_n},
    \ 
    a^{*,\alpha}_n=y^{\alpha}_{-n},
    \ 
    \ol{a}_{\alpha,n+1/2}=(-1)^{\abs{\alpha}}\frac{\pd}{\pd \psi^\alpha_{n+1/2}},
    \ 
    \ol{a}^{*,\alpha}_{n+1/2}=\psi^\alpha_{-n-1/2}
\]
for $\alpha\in\Delta^+$, $n\in\bbZ$. 

Let $I_{N,M}$ be a left ideal generated by $a^{*,\alpha}_n,\ol{a}^{*,\alpha}_{n-1/2}$ for $n\geq N$, and $a_{\alpha,m},\ol{a}_{\alpha,m+1/2}$ for $m\geq M$, $\alpha\in\Delta^+$. Since the multiplication of $\clA$ extends to its completion $\wt{\clA}$, $\wt{\clA}$ forms a topological algebra. $\wt{\clA}$ is naturally isomorphic to $U(M)$, where $M=M_{\Delta^+}$ is the $\beta\gamma$--$bc$ system associated with $\Delta^+$ (\cref{prp:mode alg of beta gamma bc sys}).

Let $\wt{\clA}_0$ be a supercommutative subalgebra of $\wt{\clA}$ topologically generated by $a^{*,\alpha}_n$ and $\ol{a}^{*,\alpha}_{n-1/2}$ ($\alpha\in\Delta^+$, $n\in\bbZ$). Let $\wt{\clA}_{\leq 1}$ be an $\wt{\clA}_0$-submodule of $\wt{\clA}$ topologically generated by $\wt{\clA}_0$ and $a_{\alpha,n}$, $\ol{a}_{\alpha,n+1/2}$ ($\alpha\in\Delta^+$, $n\in\bbZ$).  

The following claim is proved in \cref{ss:geom interpretation}.
\begin{prp}\label{prp:geom interpretation of completed Weyl alg}
    For the completed Weyl algebra $\wt{\clA}$, the following statements hold: 
    \begin{enumerate}
        \item 
        $\wt{\clA}_{\leq 1}$ is a Lie subsuperalgebra of $\wt{\clA}$,
        \item 
        $\wt{\clA}_0$ is an ideal of $\wt{\clA}_{\leq 1}$, and $\wt{\clA}_0\simeq \Fun(\clL^{N=1}U)$,
        \item 
        $\wt{\clA}_{\leq 1}/\wt{\clA}_0\simeq \Vect(\clL^{N=1}U)$.
    \end{enumerate}
\end{prp}

\subsection{Local extension}
We define Lie subsuperalgebras of $\wt{\clA}$ by  $\clA_\loc\ceq\wt{\clA}\cap \Lie(M)$, $\clA_{\loc,0}\ceq\wt{\clA}_0\cap\Lie(M)$, and $\clA_{\loc,\leq 1}\ceq \wt{\clA}_{\leq 1}\cap \Lie(M)$. Furthermore, $\clT_\loc$ denotes the image of $\clA_{\loc,\leq 1}$. By definition, we have an exact sequence:
    \[
    \begin{tikzcd}
            0 \ar[r] & \clA_{\loc,0} \ar[r] & \clA_{\loc,\leq 1} \ar[r] & \clT_\loc \ar[r] & 0
        \end{tikzcd}
    \]
By \cref{lem:superloop group and superloop Lie superalgebra}, we already have a homomorphism $\clL^{N=1}\rho\colon\clL^{N=1}\frg\to\clT_{\loc}$ of Lie superalgebras. The first step is to construct a homomorphism of Lie superalgebras $w\colon\clL^{N=1}\frg\to\clA_{\loc,\leq 1}$ lifting $\clL^{N=1}\rho$.

The following lemma is a special case of the Wick formula. 
\begin{lem}
        For $P^\beta,Q^\beta\in\bbC[a^{*,\alpha}_{n}]_{\alpha\in\Delta^+,n\leq 0}\subset M$, $\beta\in\Delta_+$, we set superfields $P^\beta(Z)\ceq Y(P^\beta,Z),Q^\beta(Z)\ceq Y(Q^\beta,Z)$. Then, we have an identity
        \begin{align*}
            \nop{P^\beta(Z)\ol{a}_\beta(Z)}\nop{Q^\gamma(W)\ol{a}_\gamma(W)}
            =&
            \nop{P^\beta(Z)\ol{a}_\beta(Z) Q^\gamma(W)\ol{a}_\gamma(W)}
            \pm\sum_{n\in\bbZ_{\geq0}}(Z-W)^{-n|1}\nop{P^\beta(Z)\frac{\pd Q^\gamma}{\pd a^{*,\beta}_{-n}}(W)\ol{a}_\gamma(W)}\\
            &\quad 
            \pm\sum_{n\in\bbZ_{\geq 0}}(Z-W)^{-n|1}\nop{\frac{\pd P^\beta}{\pd a^{*,\gamma}_{-n}}(Z)\ol{a}_{\beta}(Z)Q^\gamma(W)},
        \end{align*}
        where the sign conventions follow those of the ordinary Wick formula (see, for example \cite{Kac98}). 
    \end{lem}

Thus, we obtain the desired lift $w\colon\clL^{N=1}\frg\to\clA_{\loc,\leq 1}$ by taking normally ordered products. In other words,
\begin{prp}\label{prp:lifting by taking nop}
    There exists a SUSY vertex algebra homomorphism $w\colon V_{N=1}^0(\frg)\to M_{\Delta^+}$ such that
    \begin{align*}
        \ol{e_\alpha}(Z)
        &\lmto \ol{a}_{\alpha}(Z)+\sum_{\beta\in\Delta^+\cap(\Delta^++\alpha)}(-1)^{\abs{P^\beta_\alpha}}\nop{P_\alpha^\beta(a^*(Z))\ol{a}_\beta(Z)},\\
        \ol{h_i}(Z)
        &\lmto -\sum_{\beta\in\Delta^+}(-1)^{\beta}\beta(h_i)\nop{a^{*\beta}(Z)\ol{a}_\beta(Z)},\\
        \ol{f_\alpha}(Z)
        &\lmto\sum_{\beta\in\Delta^+}(-1)^{\abs{Q_\alpha^\beta}}\nop{Q_\alpha^\beta(a^*(Z))\ol{a}_\beta(Z)},
    \end{align*}
    where $P^\beta_\alpha,Q^\beta_\alpha$ are the polynomials in \cref{action of Lie algebra}.
\end{prp}

On the other hand, we already have the homomorphism of Lie superalgebras $\rho_{\frn_-}\colon\frg\to\Vect(U)\oplus\Fun(U)\otimes\frh$, and hence $\clL^{N=1}\rho_{\frn_-}\colon\clL^{N=1}\frg\to\clT_\loc\oplus\clA_{\loc,0}\wh{\otimes}\clL^{N=1}\frh$, where $\wh{\otimes}$ denotes the completed tensor product. Since the sequence
\[
    \begin{tikzcd}
            0 \ar[r] & \bbC \ar[r] & \wh{\frh}^\kappa \ar[r] & \clL^{N=1}\frh \ar[r] & 0
    \end{tikzcd}
\]
is exact by definition, we have an exact sequence
\[
    \begin{tikzcd}
            0 \ar[r] & \clA_{\loc,0} \ar[r] & \clA_{\loc,\leq 1}\oplus \clA_{\loc,0}\wh{\otimes}\wh{\frh}^{\kappa} \ar[r] & \clT_\loc\oplus\clA_{\loc,0}\wh{\otimes}\clL^{N=1}\frh\ar[r] & 0.
    \end{tikzcd}
\]

The next step is to construct a homomorphism of Lie superalgebras $\wh{\frg}_{N=1}^\kappa\to\clA_{\leq 1,\loc}\oplus \clA_{0,\loc}\wh{\otimes}\wh{\frh}^{\kappa}$ such that $K\mto 1$ lifting $\clL^{N=1}\rho_{\frn_-}$. By taking normally ordered products, we have a linear map $\wt{w}_\kappa\colon\clL^{N=1}\frg\to\clA_{\loc,\leq 1}\oplus\clA_{\loc,0}\wh{\otimes}\wh{\frh}^\kappa$ such that
\begin{align*}
        \wt{w_\kappa}(\ol{e}_\alpha(Z))
        &=\ol{a}_{\alpha}(Z)+\sum_{\beta\in\Delta^+\cap(\Delta^++\alpha)}(-1)^{\abs{P^\beta_\alpha}}\nop{P_\alpha^\beta(a^*(Z))\ol{a}_\beta(Z)},\\
        \wt{w_\kappa}(\ol{h}_i(Z))
        &= -\sum_{\beta\in\Delta^+}(-1)^{\abs{\beta}}\beta(h_i)\nop{a^{*\beta}(Z)\ol{a}_\beta(Z)}+\ol{b}_i(Z),\\
        \wt{w_\kappa}(\ol{f}_\alpha(Z))
        &=\sum_{\beta\in\Delta^+}(-1)^{\abs{Q_\alpha^\beta}}\nop{Q_\alpha^\beta(a^*(Z))\ol{a}_\beta(Z)}+\sum_{i=1}^l(-1)^{\abs{R^i_\alpha}}\nop{R^i_\alpha(a^*(Z))}\ol{b}_i(Z),
\end{align*}
which is a lift of $\clL^{N=1}\frg\to\clT_\loc\oplus\clA_{\loc,0}\wh{\otimes}\clL^{N=1}\frh$. In contrast to \cref{prp:lifting by taking nop}, $\wt{w}_\kappa$ is not a homomorphism of Lie superalgebras. We define a linear map $\omega_\kappa\colon\bigwedge^2\clL^{N=1}\frg\to\clA_{\loc,0}$ by 
\[
\omega_\kappa(\ol{x}(Z),\ol{y}(W))\ceq [\wt{w}_\kappa(\ol{x}(Z)),\wt{w}_\kappa(\ol{y}(W))]-\wt{w}_\kappa([\ol{x}(Z),\ol{y}(W)]).
\]
It is easy to check that $\omega_\kappa$ is an even $2$-cocycle of the Chevalley--Eilenberg cochain complex. Thus, we only have to show that $\omega_\kappa$ and $\sigma_\kappa$ are equivalent in $H^2(\clL^{N=1}\frg,\clA_{\loc,0})$. Indeed, if there exists a $\gamma\colon\clL^{N=1}\frg\to\clA_{\loc,0}$ such that $\sigma_\kappa=\omega_\kappa+d\gamma$, then the linear map $w_\kappa\ceq\wt{w}_\kappa+\gamma\colon\wh{\frg}^\kappa\to \clA_{\loc,\leq 1}\oplus\clA_{\loc,0}\wh{\otimes}\frh^\kappa$ is a homomorphism of Lie superalgebras, which is the desired lift.


\section{Comparison of cohomology classes}\label{s:Comparison of cohomology classes}

    In the previous section, normal ordering produces an even $2$-cocycle with coefficients in $\clA_{\loc,0}$. The purpose of this section is to show that its cohomology class is determined by its restriction to the Cartan subalgebra. This section is the SUSY analogue of \cite[\S 3]{Fre05}.

    \subsection{The local Chevalley--Eilenberg cochain complex}
    Let $\clC=\clC(\clL_{N=1}\frg\oplus(\clL_{N=1}\frg)^*)$ be a Clifford algebra on $\clL_{N=1}\frg\oplus(\clL_{N=1}\frg)^*$, i.e.\ for a fixed basis $\{x_a\}$ of $\frg$, and its dual basis $\{x^a\}$ of $\frg^*$, $\clC$ is topologically generated by $\psi_{a,n},\ol{\psi}_{a,n+1/2},\psi^a_n,\ol{\psi}^a_{n-1/2}$ ($n\in\bbZ$) with parity $\abs{\psi_{a,n}}=\abs{\psi^a_n}=\abs{x_a}+1$, $\abs{\ol{\psi}_{a,n}}=\abs{\ol{\psi}^a_n}=\abs{x_a}$ satisfying the following relations:
    \[
        [\psi_{a,n},\psi^b_m]=\delta_{n+m,0}\delta^b_a,
        \quad
        [\ol{\psi}_{a,n},\ol{\psi}^b_m]=\delta_{n+m,0}\delta^b_a.
    \]
    Throughout this section, we write $\abs{a}\ceq\abs{x_a}=\abs{x^a}$.
    
    The completed supersymmetric algebras $\Sym(\Pi(\clL^{N=1}\frg))$ and $\Sym(\Pi(\clL^{N=1}\frg^*))$ are subalgebras of the Clifford algebra $\clC$. We define $\clC$-modules $\Wedge(\clL^{N=1}\frg)$ and $\Wedge(\clL^{N=1}\frg^*)$ by
    \[
        \Wedge(\clL^{N=1}\frg) \ceq \clC\otimes_{\Sym(\Pi(\clL^{N=1}\frg^*))}\bbC,
        \quad 
        \Wedge(\clL^{N=1}\frg^*) \ceq \clC\otimes_{\Sym(\Pi(\clL^{N=1}\frg))}\bbC,
    \]
    where $\bbC$ is a trivial module. Note that $\Wedge(\clL^{N=1}\frg)$ and $\Wedge(\clL^{N=1}\frg^*)$ are isomorphic to $\Sym(\Pi(\clL^{N=1}\frg^*))$ and $\Sym(\Pi(\clL^{N=1}\frg))$ as a $\Sym(\Pi(\clL^{N=1}\frg^*))$- and $\Sym(\Pi(\clL^{N=1}\frg))$-module, respectively.

    Since we can identify the completed tensor product $\Wedge(\clL^{N=1}\frg^*)\wh{\otimes}\clA_{\loc,0}$ with $\Hom^\cont(\Wedge(\clL^{N=1}\frg),\clA_{\loc,0})$, we regard $\Wedge(\clL^{N=1}\frg^*)\wh{\otimes}\clA_{\loc,0}$ as the Chevalley--Eilenberg cochain complex $C^\bl(\clL^{N=1}\frg,\clA_{\loc,0})$. Then, the differential is expressed as
    \begin{align*}
        d
        &\ceq\oint dz
        \Biggl(
        \sum_a \Bigl( (-1)^{\abs{a}}\psi^a(z)\clL^{N=1}\rho (x_a(z))+(-1)^{\abs{a}+1}\ol{\psi}^a(z)\clL^{N=1}\rho(\ol{x}_a(z)) \Bigr)  \\
        &\phantom{\ceq \oint dz \Bigl(} 
         -\frac{1}{2}\sum_{a,b,c} \Bigl(
          (-1)^{\abs{a}\abs{c}}\mu^c_{a,b}\nop{\psi^a(z)\psi^b(z)\psi_c(z)} \\
        &\phantom{\ceq \oint dz \Bigl(-\frac{1}{2}\sum_{a,b,c} \bigl(} 
        +(-1)^{(\abs{a}+1)(\abs{c}+1)}\mu^c_{a,b}\nop{\ol{\psi}^a(z)\psi^b(z)\ol{\psi}_c(z)}+(-1)^{\abs{a}\abs{c}}\mu^c_{a,b}\nop{\psi^a(z)\ol{\psi}^b(z)\ol{\psi}_c(z)})
        \Bigr) \Biggr)
        \\
        &=\oint dZ
        \left(
        \sum_a \Bigl(-\psi^a(Z)\clL^{N=1}\rho(\ol{x}_a(Z)) \Bigr)
        -\sum_{a,b,c}(-1)^{(\abs{a}+1)\abs{c}}\frac{1}{2}\mu_{a,b}^c\nop{\psi^a(Z)\psi^b(Z)\ol{\psi}_c(Z)}
        \right),
    \end{align*}
    where $\mu_{a,b}^c\in\bbC$ are the structure constants of $\frg$, i.e.\  $[x_a,x_b]=\sum_c\mu_{a,b}^cx_c$. Here, $\psi_a(z)$, $\ol{\psi}_a(z)$, $\psi^a(z)$, $\ol{\psi}^a(z)$ are formal series with coefficients in $\clC$ defined by
    \[
        \psi_a(z)\ceq\sum z^{-n-1}\psi_{a,n},
        \quad
        \ol{\psi}_a(z)\ceq\sum z^{-n-1}\ol{\psi}_{a,n+1/2},
        \quad
        \psi^a(z)\ceq\sum z^{-n}\psi^a_n,
        \quad
        \ol{\psi}^a(z)\ceq\sum z^{-n}\ol{\psi}^a_{n-1/2}, 
    \]
    and $\ol{\psi}_a(Z)$ and $\psi^a(Z)$ are formal superseries defined by 
    \begin{equation}\label{eq:superfield psi}
        \ol{\psi}_a(Z)=\ol{\psi}_a(z)+\theta\psi_a(z),
        \quad
        \psi^a(Z)=\psi^a(z)+(-1)^{\abs{a}+1}\theta\ol{\psi}^a(z).
    \end{equation} 

    Note that $C^\bl(\clL^{N=1}\frg,\wt{\clA}_{\loc,0})=\Wedge(\clL^{N=1}\frg^*)\wt{\otimes}\clA_{\loc,0}$ has a $\bbZ\Delta\oplus\frac{1}{2}\bbZ$-grading such that $\deg(\psi^\alpha_n)=(-\alpha,-n)$, $\deg(\ol{\psi}^\alpha_m)=(-\alpha,-m)$, $\deg(a^{*,\alpha}_n)=(-\alpha,-n)$, $\deg(\ol{a}^{*,\alpha}_m)=(-\alpha,-m)$. Since the gradation of $d$ is $(0,0)$, the $\bbZ\Delta\oplus\frac{1}{2}\bbZ$-gradation is induced on the cohomology groups $H^\bl(\clL^{N=1}\frg,\wt{\clA}_{\loc,0})$.

    In order to define the local Chevalley--Eilenberg cochain complex, we need some preparation.
    \begin{dfn}
        Since $\clC$ is isomorphic to the SUSY Weyl superalgebra associated with $\Pi\Delta$, we can define a $\beta\gamma$--$bc$ system denoted by $\Wedge$, strongly generated by $\psi^a(Z)$, $\ol{\psi}_a(Z)$ in \cref{eq:superfield psi}. Let ${\Wedge}_+$ be a sub-SUSY VA generated by $\psi^a(Z)$, and  $M_+$ be a sub-SUSY VA generated by $a^{*,\alpha}(Z)$. Clearly,
        \[
            M_+
            =\bbC[a^{*,\alpha}_{n},\ol{a}^*_{n-1/2}]_{n\leq 0}
            =\bbC[y^\alpha_{n},\psi^\alpha_{n+1/2}]_{n\geq 0}
            =\Fun(\clJ^{N=1}U),
        \]
        where $\clJ^{N=1}U$ is a superarc space, i.e.\ $\clJ^{N=1}U(S)=U(S\dbr{T})$ for any supercommutative algebra $S$. We have a natural homomorphism of Lie superalgebras $\clJ^{N=1}\rho\colon \clJ^{N=1}\frg=\frg\dbr{T}\to\Vect(\clJ^{N=1}U)$. 
        
        Then, $M_+\otimes {\Wedge}_+$ is closed under the action of
        \[
            d\ceq \oint dZ
            \left(
            \sum_a
            (-\psi^a(Z)\clJ^{N=1}\rho(\ol{x}_a(Z)))-\sum_{a,b,c}(-1)^{(\abs{a}+1)\abs{c}}\frac{1}{2}\mu_{a,b}^c\nop{\psi^a(Z)\psi^b(Z)\ol{\psi}_c(Z)}
            \right),
        \]
        which defines a cochain complex structure on $M_+\otimes {\Wedge}_+$, and the resulting complex is isomorphic to the Chevalley--Eilenberg cochain complex $C^\bl(\clJ^{N=1}\frg,M_+)$.
    \end{dfn}

    \begin{lem}\label{lem:anothor complex}
        The map $\oint dZY(-,Z)\colon C^\bl(\clJ^{N=1}\frg,M_+)\to C^\bl(\clL^{N=1}\frg,\clA_{\loc,0})$ is a homomorphism of cochain complexes. 
    \end{lem}

    \begin{dfn}
        We denote the image of the homomorphism in \cref{lem:anothor complex} by $C^\bl_{\loc}(\clL^{N=1}\frg,\clA_{\loc,0})$, and denote its cohomology by $H^\bl_{\loc}(\clL^{N=1}\frg,\clA_{\loc,0})$. We call them the \emph{local Chevalley--Eilenberg complex} and the \emph{local Chevalley--Eilenberg cohomology}, respectively.
    \end{dfn}

    \begin{lem}
        The even $2$-cocycles $\omega_\kappa$, $\sigma_\kappa$ are contained in the local Chevalley--Eilenberg cochain complex.
    \end{lem}

    \begin{proof}
    Let 
    \[
        \wt{w}_\kappa(\ol{x}_a(Z))=\sum_{\alpha}(-1)^{\abs{P^\alpha_a}}\nop{P^\alpha_a(a^*(Z))\ol{a}_\alpha(Z)}+\sum_i(-1)^{\abs{R^i_a}}R^i_a(a^*(Z))\ol{b}_i(Z)
    \]
    for $a=1,\ldots,\dim\frg$. Then, $\omega_\kappa$ is written as
    \begin{align*}
        \omega_\kappa(\ol{x}_a(Z),\ol{x}_b(W))
        =&\sum_{i,j}(-1)^{\abs{R^i_a}}\kappa(h_i,h_j)
        (\delta(Z,W)(\pd_W^{(0|1)}R^i_a(a^*(W)))R^j_b(a^*(W))\\
        &+(\pd_W^{(0|1)}\delta(Z,W))R^i_a(a^*(W))R^j_b(a^*(W))).
    \end{align*}
Thus, we have
    \begin{align*}
        \omega_\kappa
        =\sum_{i,j}(-1)^{\abs{R^i_a}}\kappa(h_i,h_j)\oint dZ&\sum_{a,b}(\psi^a(Z)\psi^b(Z)(\pd_Z^{(0|1)}R^i_a(a^*(Z)))R^j_b(a^*(Z))\\
        &+(\pd_Z^{(0|1)}\psi^a(Z))\psi^b(Z)R^i_a(a^*(Z))R^j_b(a^*(Z))).
    \end{align*}

On the other hand, $\sigma_\kappa$ is written as
    \[
        \sigma_\kappa=(-1)^{\abs{x_a}}\oint dZ \sum_{a,b}(\pd_Z^{(0|1)}\psi^a(Z))\psi^b(Z)\kappa(x_a,x_b).
    \]
    \end{proof}

    \begin{lem}\label{lem:local complex and another complex}
        The kernel of the map $\oint dZ\ Y(-,Z)$ is $\Img S+\bbC$, and the kernel of $S$ is $\bbC$. In other words, the second page of the spectral sequence of the double complex
        \[
            \begin{tikzcd}
                0 \ar[r] &\bbC \ar[r] & C^\bl(\clJ^{N=1}\frg,M_+) \ar[r,"S"] &C^\bl(\clJ^{N=1}\frg,M_+) \ar[r] & \bbC \ar[r] & 0
            \end{tikzcd}    
        \]
        is isomorphic to $H^\bl_\loc(\clL^{N=1}\frg,\clA_{\loc,0})$.
    \end{lem}

    \begin{proof}
        We prove that the kernel of $\oint dZ Y(-,Z)$ is equal to $\Img S+\bbC$. We can prove that $U(M_+\otimes{\Wedge}_+)$ is isomorphic to $\clA_{\loc,0}\wh{\otimes}\Wedge(\clL^{N=1}\frg)$ by \cref{prp:mode alg of beta gamma bc sys}. Since the natural map $\Lie(M_+\otimes{\Wedge}_+)\to U(M_+\otimes{\Wedge}_+)\simeq\clA_{\loc,0}\wh{\otimes}\Wedge(\clL_{N=1}\frg)$ is injective (\cref{prp:emb of current LSA}), the kernel of $\oint dZ Y(-,Z)$ is the intersection of $(M_+\otimes{\Wedge}_+)$ and  $\Img(S\otimes 1+1\otimes\sd_z)$ in $(M_+\otimes{\Wedge}_+)\dpr{Z}$, that is $\Img S+\bbC$.
    \end{proof}

    \subsection{Reduction on the Cartan subalgebras}
    In this subsection, we compute the Chevalley--Eilenberg cohomology $H^\bl(\clJ^{N=1}\frg,M_+)$. Since $M_+$ is isomorphic to $\Fun(\clJ^{N=1}U)$, we have 
    \[
        M_+\simeq\Coind^{\clJ^{N=1}\frg}_{\clJ^{N=1}\frb_-}(\bbC)=\bigoplus_{(\alpha,n)}\Hom_\bbC(U(\clJ^{N=1}\frn_+)_{(\alpha,n)},\bbC)\subset \Hom_{\clJ^{N=1}\frb_-}(U(\clJ^{N=1}\frg),\bbC) 
    \]
    as $\clJ^{N=1}\frg$-module, where the index $(\alpha,n)$ runs over $\bbZ_{\geq 0}\Delta^+\oplus\frac{1}{2}\bbZ_{\geq 0}$ and $\bbC$ is a trivial $\clJ^{N=1}\frb_-$-module. Thus, by Shapiro's lemma \cite[Chapter~1,\ \S~5.4]{Fuks86}, we obtain the following lemma:
    \begin{lem}
        The restriction map induces a quasi-isomorphism $C^\bl(\clJ^{N=1}\frg,M_+)\to C^\bl(\clJ^{N=1}\frb_-,\bbC)$.
    \end{lem}

    \begin{lem}
        The restriction map $C^\bl(\clJ^{N=1}\frb_-,\bbC)\to C^\bl(\clJ^{N=1}\frh,\bbC)$ is a quasi-isomorphism.
    \end{lem}

    \begin{proof}
        Consider the Hochschild--Serre spectral sequence:
    \[
        E^{p,q}_2=H^p(\clJ^{N=1}\frh,H^q(\clJ^{N=1}\frn_-,\bbC))\Longrightarrow H^{p+q}(\clJ^{N=1}\frb_-,\bbC).
    \]
    Since $\frh\subset\clJ^{N=1}\frh$ defines an inner grading (\cite[Chapter~1,\ \S~5.2]{Fuks86}) on $\clJ^{N=1}\frh$ and $H^q(\clJ^{N=1}\frn_-,\bbC)$, we have
    \[
        H^p(\clJ^{N=1}\frh,H^q(\clJ^{N=1}\frn_-,\bbC))=H^p(\clJ^{N=1}\frh,H^q(\clJ^{N=1}\frn_-,\bbC)_0).
    \]
    By construction, we have $H^q(\clJ^{N=1}\frn_-,\bbC)_0=\delta_{q,0}\bbC$. Thus, the spectral sequence collapses at the $E_2$-page. Thus, we have $H^p(\clJ^{N=1}\frb_-,\bbC)\simeq H^p(\clJ^{N=1}\frh,\bbC)$.
    \end{proof}

    Combining the above lemmata, we have
    
    \begin{lem}\label{lem:combing quis}
        The map $\mu\colon C^\bl(\clJ^{N=1}\frg,M_+)\to C^\bl(\clJ^{N=1}\frh,\bbC)$, induced by the natural projection $M_+\to\bbC$, and the restriction $\clJ^{N=1}\frh\to\clJ^{N=1}\frb_-$, yields a quasi-isomorphism.  
    \end{lem}
    
    \begin{lem}\label{lem:isom}
        There is an isomorphism of cohomology groups:
        \[
            H^\bl(\clJ^{N=1}\frh,\bbC)/(\Img S+\bbC)\simeq H^\bl_\loc(\clL^{N=1}\frg,\clA_{\loc,0}).
        \]
    \end{lem}

    \begin{cor}\label{cor:the comparison of local cocycle}
        If a cocycle $\varphi$ of $C^\bl_\loc(\clL^{N=1}\frg,\clA_{0,\loc})$ vanishes on $\clL^{N=1}\frh$, then the cohomology class of $\varphi$ is equal to $0$.         
    \end{cor}

    \begin{proof}
        Assume a cocycle $\varphi\in C^\bl_\loc(\clL^{N=1}\frg,\clA_{0,\loc})$ vanishes on $\clL^{N=1}\frh$. By \cref{lem:local complex and another complex}, there exists a cocycle $A\in C^\bl(\clJ^{N=1}\frg,M_+)$ such that $\oint dZ\ Y(A,Z)=\varphi$. The restriction of $A$ to $\clJ^{N=1}\frh$ is denoted by $\ol{A}$. Then, since $\oint dZ\ Y(\ol{A},Z)=\varphi|_{\clL^{N=1}\frh}=0$, we have $\ol{A}\in\Img S$. Since $\mu$ preserves $S$, we have $\mu(A)=S(h)$ for some $h\in C^\bl(\clJ^{N=1}\frh,\bbC)$. Since $S$ and $d$ commute, and $\mu(A)$ is a cocycle, we have $S(dh)=0$. Furthermore, since the kernel of $S$ is $\bbC\ket{0}=C^0(\clJ^{N=1}\frh,\bbC)$, $dh$ must be zero, that is, $h$ is a cocycle. Then, by \cref{lem:combing quis}, there exists some cocycle $B\in C^\bl(\clJ^{N=1}\frg,M_+)$ which is equivalent to $h$. Thus, $A$ and $SB$ are equivalent, and hence the cohomology class of $\varphi$ is equal to $0$. 
    \end{proof}

\section{Free field realization of SUSY affine vertex algebras}

In this section, we construct the free-field homomorphism $w_\kappa$, compute the image of the Kac--Todorov superconformal vector, and determine when $w_\kappa$ is injective. Apart from the injectivity analysis, this section is the supersymmetric analogue of \cite[\S~5]{Fre05}. This analysis is specific to the supersymmetric setting: unlike its ordinary affine counterpart, the injectivity of $w_\kappa$ depends on the level $\kappa$.

\subsection{Free field realization of SUSY affine vertex algebras}
    \begin{thm}\label{thm:free field realizations}
        There exists a unique homomorphism of SUSY vertex algebras
        \[
            w_\kappa\colon V^\kappa_{N=1}(\frg)\to M_{\Delta_+}\otimes\Pi^\kappa
        \]
        such that
        \begin{align*}
            w_\kappa(\ol{e_i}(Z))
            &= \ol{a}_{\alpha_i}(Z)+\sum_{\beta\in\Delta^+,\beta-\alpha_i\in\Delta^+}(-1)^{\abs{P^\beta_i}}\nop{P_i^\beta(a^*(Z))\ol{a}_\beta(Z)},\\
            w_\kappa(\ol{h_i}(Z))
            &= -\sum_{\beta\in\Delta^+}(-1)^{\abs{\beta}}\beta(h_i)\nop{a^{*\beta}(Z)\ol{a}_\beta(Z)}+\ol{b}_i(Z),\\
            w_\kappa(\ol{f_i}(Z))
            &=\sum_{\beta\in\Delta^+}(-1)^{\abs{Q_i^\beta}}\nop{Q_i^\beta(a^*(Z))\ol{a}_\beta(Z)}-\kappa(e_i,f_i)\pd_Z^{(0|1)} a^{*,\alpha_i}(Z)+(-1)^{\abs{\alpha_i}}a^{*,\alpha_i}(Z)\ol{b}_i(Z)
        \end{align*}
    \end{thm}

    \begin{proof}
        We have defined two even $2$-cocycles $\omega_\kappa$ and $\sigma_\kappa$, which are in $C^2_{\loc}(\clL^{N=1}\frg,\clA_{\loc,0})$. A direct computation gives
        \[
            \omega_\kappa(\ol{h}_i(Z),\ol{h}_j(W))=[\ol{b}_i(Z),\ol{b}_j(W)]=\pd_W^{(0|1)}\delta(Z,W)k(h_i,h_j)=\sigma_\kappa(\ol{h}_i(Z),\ol{h}_j(W)).
        \]
        Thus, by \cref{cor:the comparison of local cocycle}, $\omega_\kappa$ and $\sigma_\kappa$ are equivalent in $H^2_{\loc}(\clL^{N=1}\frg,\clA_{\loc,0})$. Equivalently, there exists $\gamma\in C^1_\loc(\clL^{N=1}\frg,\clA_{\loc,0})$ such that $\sigma_\kappa=\omega_\kappa+d\gamma$. Since both $\omega_\kappa$ and $\sigma_\kappa$ have grading $(0,0)$, so does $\gamma$. Thus, we have
        \[
            \gamma(\ol{e}_i(Z))=0,
            \quad
            \gamma(\ol{h}_i(Z))=0,
            \quad
            \gamma(\ol{f}_i(Z))=c_i\pd_Z^{(0|1)} a^{*,\alpha_i}(Z)
        \]
        for some $c_i\in\bbC$. By an explicit computation, we have
        \[
            \omega_\kappa(\ol{e}_i(Z),\ol{f}_i(W))=0,
            \quad
            d\gamma(\ol{e}_i(Z),\ol{f}_i(W))=(-1)^{\abs{\alpha_i}+1}c_i\pd_W^{(0|1)}\delta(Z,W).
        \]
        Thus, we find $c_i=-\kappa(e_i,f_i)$, which completes the proof.
        \end{proof}

        \begin{dfn}
            For $\lambda\in\frh^*$, we define a $\frh^\kappa$-module $\Pi^\kappa_\lambda$ by
            $\Pi^\kappa_\lambda\ceq U(\wh{\frh}^\kappa)\otimes_{U(\frh\dbr{T}\oplus \bbC K)}\bbC\ket{\lambda}$, where $\bbC\ket{\lambda}$ is an $\frh\dbr{T}\oplus \bbC K$-module defined by $b_{i,n}\ket{\lambda}=\ol{b}_{i,n-1/2}\ket{\lambda}=0$ ($n>0$), $K\ket{\lambda}=\ket{\lambda}$, $b_{i,0}\ket{\lambda}=\lambda(h_i)\ket{\lambda}$. By \cref{thm:free field realizations},  $W^\kappa_\lambda\ceq M_{\Delta^+}\otimes\Pi^\kappa_\lambda$ is a $\wh{\frg}^\kappa_{N=1}$-module.   
        \end{dfn}

    \subsection{$N=1$ superconformal structure}

    In this subsection, we investigate an $N=1$ superconformal structure. Recall the $N=1$ superconformal structure on SUSY affine vertex algebras introduced in \cite{KT85} (see also \cite{HK07}).

    \begin{prp}[\cite{KT85}]
        Let $\frg$ be a Lie superalgebra with a nondegenerate supersymmetric invariant form $\kappa$. Let $\{x_i\}$ be a basis of $\frg$, and let $\{x^i\}$ be the dual basis with respect to $\kappa$, i.e.\ $\kappa(x_i,x^j)=\delta_i^j$. 
        \[
            \tau\ceq \left(\sum_i(-1)^{\abs{x_i}}x_{i,-1}\ol{x}^i_{-1/2}+\frac{1}{3}\sum_{i,j}(-1)^{\abs{x_i}}\ol{x}_{i,-1/2}\ol{x}_{j,-1/2}\ol{[x^i,x^j]}_{-1/2}\right)\ket{0}
        \]
        gives an $N=1$ superconformal vector in $V^\kappa_{N=1}(\frg)$. Thus, $G^\frg(Z)=G^\frg(z)+2\theta L^\frg(z)\ceq Y(\tau,Z)$ satisfies the $N=1$ super Virasoro algebra relations (\cref{dfn:N=1 super Virasoro algebra}). Then, $x_{-1}\ket{0}$ (resp.\ $\ol{x}_{-1/2}\ket{0}$) is a primary vector of conformal weight $1$ (resp.\ $1/2$). That is, for $x\in\frg$, 
        \[
            G^\frg(Z)\ol{x}(W)
            \sim 2(Z-W)^{-1|1}\pd_W^{(1|0)}\ol{x}(W)
            +(Z-W)^{-1|0}\pd_W^{(0|1)}\ol{x}(W)
            +(Z-W)^{-2|1}\ol{x}(W).
        \]
    \end{prp}
    
    \begin{prp}\label{prp:superconformal vect}
        If $\kappa$ is nondegenerate, the image of $G^\frg(Z)$ under $w_\kappa$ is 
        \begin{align*}
            w_\kappa(G^\frg(Z))
            =&\sum_{\alpha\in\Delta^+}((-1)^{\abs{\alpha}}\nop{\ol{a}_{\alpha}(Z)\pd_Z^{(1|0)}a^{*,\alpha}(Z)}+\nop{(\pd_Z^{(0|1)}\ol{a}_{\alpha}(Z))(\pd_Z^{(0|1)}a^{*,\alpha}(Z))})\\
            &-\sum_i\nop{(\pd_Z^{(0|1)}\ol{b_i}(Z))\ol{b}^i(Z)}-2\pd_Z^{(1|0)}\ol{\rho},     
        \end{align*}
        where $\rho=\rho_{\Delta^+}$ is an element of $\frh$ corresponding to the Weyl vector via $\kappa$.  
    \end{prp}

    \begin{proof}
        Note that $V^\kappa_{N=1}(\frg)$ and $M\otimes\Pi^\kappa$ are $\frac{1}{2}\bbZ_{\geq 0}$-graded SUSY vertex algebras whose Hamiltonians $H$ are uniquely determined by
        \[
            H\ol{x}_{-1/2}\ket{0}=\frac{1}{2}\ol{x}_{-1/2}\ket{0},\quad Ha^{*,\alpha}_0\ket{0}=0,\quad H\ol{a}_{\alpha,-1/2}\ket{0}=\frac{1}{2}\ol{a}_{\alpha,-1/2}\ket{0},\quad H\ol{b}_{i,-1/2}\ket{0}=\frac{1}{2}\ol{b}_{i,-1/2}\ket{0}
        \]
        for $x\in\frg$, $\alpha\in\Delta^+$, $i=1,\ldots l$, and $w_\kappa$ commutes with $H$. Furthermore, $V^\kappa_{N=1}(\frg)$ and $M\otimes\Pi^\kappa$ are naturally endowed with $\bbZ\Delta$-gradings, respectively, and $w_\kappa$ also preserves these gradings. Thus, the corresponding vector is a linear combination of
        \begin{equation}\label{eq:candidate of superconformal vector I}
            \ol{b}_{i,-1/2}\ol{b}_{j,-1/2}\ol{b}_{k,-1/2},
            \quad 
            b_{i,-1}\ol{b}_{j,-1/2},
            \quad 
            \ol{b}_{i,-3/2},
        \end{equation}
        \begin{equation}\label{eq:candidate of superconformal vector II}
            \ol{a}_{\alpha,-1/2}a^{*,\alpha}_{-1},
            \quad 
            a_{\alpha,-1}\ol{a}^{*,\alpha}_{-1/2},
            \quad
            \ol{a}_{\alpha,-3/2}a^{*,\alpha}_0,
        \end{equation}
        \begin{equation}\label{eq:candidate of superconformal vector III}
            \ol{a}_{\alpha,-1/2}\ol{a}_{\beta,-1/2}\ol{a}_{\gamma,-1/2}a^{*,\alpha}_0a^{*,\beta}_0a^{*,\gamma}_0,
            \quad
            a_{\alpha,-1}\ol{a}_{\beta,-1/2}a^{*,\alpha}_0a^{*,\beta}_0,
            \quad
            \ol{a}_{\alpha,-1/2}a^{*,\alpha}_0\ol{b}_{i,-1/2}\ol{b}_{j,-1/2}.
        \end{equation}
        For simplicity, we omit $w_\kappa$. If we prove that $\ol{a}_{\alpha,-1/2}\ket{0}$ (resp.\ $a^{*,\alpha}_0\ket{0}$) is a primary vector of conformal weight $1/2$ (resp.\ $0$) with respect to $G^\frg(Z)$, then $G^\frg(Z)$ is  
        \[
            G^{M}(Z)=\sum_{\alpha\in\Delta^+}((-1)^{\abs{\alpha}}
            \nop{\ol{a}_{\alpha}(Z)\pd_Z^{(1|0)}a^{*,\alpha}(Z)}
            +\nop{(\pd_Z^{(0|1)}\ol{a}_{\alpha}(Z))(\pd_Z^{(0|1)}a^{*,\alpha}(Z))}         
        \]
        plus a linear combination of \cref{eq:candidate of superconformal vector I}. 
        
        We first prove that $a^{*,\alpha}_0\ket{0}$ is primary of conformal weight $0$ with respect to $G^\frg(Z)$. Since $L^\frg_{n}$ and $G^\frg_{m}$ have $H$-gradings $n\in\bbZ$ and $m\in\bbZ+1/2$, respectively, it is enough to show that $a^{*,\alpha}_0\ket{0}$ has $L^\frg_0$-grading $0$. Since $L^\frg_0$ preserves $\bbZ\Delta$-grading, $L^\frg_0a^{*,\alpha}_0\ket{0}$ is an element of $\bbC[a^*_0]\ket{0}$ of weight $-\alpha$. Assume that $L^\frg_0a^{*,\alpha}_0\ket{0}\neq 0$. Let $P(e_{\beta})$ be an element of $U(\frn_+)$ of weight $\alpha$ such that $P(e_{\beta,0})L^\frg_0a^{*,\alpha}_0\ket{0}=\ket{0}$. By considering the weight, $P(e_{\beta,0})a^{*,\alpha}_0\ket{0}=\lambda\ket{0}$ for some $\lambda\in\bbC$. Thus,
        \[
            \ket{0}
            =P(e_{\beta,0})L^\frg_0a^{*,\alpha}_0\ket{0}
            =L^\frg_0P(e_{\beta,0})a^{*,\alpha}_0\ket{0}
            =\lambda L^\frg_0\ket{0}
            =0,
        \]
        which is contradiction.


        We next prove that $\ol{a}_{\alpha,-1/2}$ is primary of conformal weight $1/2$. Recall that $\ol{e}_\alpha(Z)$ is of the form
        \begin{equation}\label{eq:w(e)}
            \ol{e}_\alpha(Z)=\ol{a}_\alpha(Z)+\sum_{\beta\in\Delta^+\cap(\Delta^++\alpha)}(-1)^{\abs{P^\beta_\alpha}}\nop{P^\beta_\alpha(a^*(Z))\ol{a}_\beta(Z)}.
        \end{equation}
        By the formula \cref{eq:w(e)}, we have 
        \begin{align*}
            L^\frg_n\ol{a}_{\alpha,-1/2}\ket{0}
            &=L^\frg_n(\ol{e}_{\alpha,-1/2}-\sum_{\beta\in\Delta^+\cap(\Delta^++\alpha)} (-1)^{\abs{P^\beta_\alpha}}P^\beta_\alpha(a^*_0)\ol{a}_{\beta,-1/2})\ket{0}\\
            &=L^\frg_n\ol{e}_{\alpha,-1/2}\ket{0}-\sum_{\beta\in\Delta^+\cap(\Delta^++\alpha)} (-1)^{\abs{P^\beta_\alpha}}P^\beta_\alpha(a^*_0)L^\frg_n\ol{a}_{\beta,-1/2})\ket{0},\\
            G^\frg_{n+1/2}\ol{a}_{\alpha,-1/2}\ket{0}
            &=G^\frg_{n+1/2}\ol{e}_{\alpha,-1/2}\ket{0}-\sum_{\beta\in\Delta^+\cap(\Delta^++\alpha)}P^\beta_\alpha(a^*_0)G^\frg_{n+1/2}\ol{a}_{\beta,-1/2})\ket{0}
        \end{align*}
        for $n\geq 0$. Thus, the claim follows by induction on $\abs{\Delta^+}-(\text{height of}\ \alpha)$.
        
        We proceed to determine the remaining terms of the form \cref{eq:candidate of superconformal vector I}. Recall that $\ol{h}_i(Z)$ is of the form
        \begin{equation}\label{eq:w(h)}
             \ol{h}_i(Z)=-\sum(-1)^{\abs{\beta}}\beta(h_i)\nop{a^{*,\beta}(Z)\ol{a}_\beta(Z)}+\ol{b}_i(Z).
        \end{equation}
        Using the following OPEs:
        \begin{align*}
            &\nop{\ol{a}_\alpha(Z)\pd_Z^{(1|0)}a^{*,\alpha}(Z)}
            \nop{a^{*,\beta}(W)\ol{a}_\beta(W)}\\
            &\sim
            \delta_{\alpha,\beta}((Z-W)^{-1|1}(-1)^{\abs{\beta}}\pd_W^{(1|0)}\nop{a^{*,\beta}(W)\ol{a}_\beta(W)}
            +(Z-W)^{-2|1}(-1)^\beta\nop{a^{*,\alpha}(W)\ol{a}_{\alpha}(W)}),\\
            &\nop{(\pd_Z^{(0|1)}\ol{a}_\alpha(Z))(\pd_Z^{(0|1)}a^{*,\alpha}(Z)}
            \nop{a^{*,\beta}(W)\ol{a}_\beta(W)}\\
            &\sim\delta_{\alpha,\beta}
            (-(Z-W)^{-1|0}\pd_W^{(0|1)}\nop{a^{*,\beta}(W)\ol{a}_\beta(W)}
            -(Z-W)^{-1|1}\pd_W^{(1|0)}\nop{a^{*,\beta}(W)\ol{a}_\beta(W)}
            +(Z-W)^{-2|0}),
        \end{align*}
        we have
        \begin{align*}
            G^M(Z)\ol{h}_i(W)
            \sim&
            (Z-W)^{-1|1}2\pd_W^{(1|0)}w_\kappa(\ol{h}_i(W))
            +(Z-W)^{-1|0}\pd_W^{(0|1)}w_\kappa(\ol{h}_i(W))
            +(Z-W)^{-2|1}w_\kappa(\ol{h}_i(W))\\
            &-(Z-W)^{-1|1}2\pd_W^{(1|0)}\ol{b}_i(W)
            -(Z-W)^{-1|0}\pd_W^{(0|1)}\ol{b}_i(W)
            -(Z-W)^{-2|1}\ol{b}_i(W)\\
            &+(Z-W)^{-2|0}\sum_{\alpha\in\Delta^+}(-1)^{\abs{\alpha}}\alpha(h_i).
        \end{align*}
        Thus, we determine $G^\frg(Z)$ as 
        \[
            G^\frg(Z)=G^M(Z)-\sum_i\nop{(\pd_Z^{(0|1)}\ol{b}_i(Z))\ol{b}^i(Z)}-2\pd_Z^{(1|0)}\ol{\rho}(Z)
        \]
    \end{proof}
        
    \subsection{Injectivity of $w_\kappa$}
    We recall only those aspects of supergeometry that will be used in the proof of the injectivity of $w_\kappa$. Most of the standard scheme-theoretic notions extend to superschemes in a straightforward manner (for example, see \cite{BHRP23}). Let $X$ and $Y$ be superschemes of finite type over $\bbC$.
    
    \begin{dfn}
        A morphism $f\colon X\to Y$ of relative dimension $(n|m)$ is \emph{smooth} if $f$ is a flat morphism and $\Omega_{X/Y}$ is locally free of rank $(n|m)$.  
    \end{dfn}

    Equivalently, under our finite-type assumptions, smoothness can be checked locally from the rank of the corresponding relative Jacobian matrix (for details, see \cite[Proposition~A.20]{BHRP23}).

    \begin{dfn}\label{dfn:formally smooth}
        A morphism $f\colon X\to Y$ is \emph{formally smooth} if for any supercommutative algebra $B$ and its nilpotent ideal $I$, any $Y$-morphism $\Spec(B/I)\to X$ extends $\Spec(B)\to X$. 
    \end{dfn}

    \begin{lem}[{\cite[Proposition~A.25]{BHRP23}}]\label{lem:formally smooth}
        For a morphism $f\colon X\to Y$, formal smoothness is equivalent to smoothness.        
    \end{lem}
    
    
    \begin{dfn}
        A morphism $f\colon X\to Y$ is faithfully flat if it is flat and surjective on the underlying topological space.  
    \end{dfn}

    \begin{rmk}\label{rmk:surjective}
        As in the purely even case, the surjectivity is equivalent to the following condition: for any field $K$ and $y\in Y(K)$, there exists a field extension $L$ of $K$ and $x\in X(L)$ such that $f(x)=y$. 
    \end{rmk}

    We recall the following standard criterion for faithful flatness, which will be directly used in the proof of injectivity of $w_\kappa$.  
    \begin{lem}[{\cite[Corollary~3.2]{MZ22}}]\label{lem:faithfully flat}
        When $X$ and $Y$ are affine superschemes, a morphism $f\colon X\to Y$ is faithfully flat if and only if the homomorphism $f^*\colon\Fun(Y)\to\Fun(X)$ is faithfully flat, i.e.\ the functor from the category of $\Fun(Y)$-modules to the category of $\Fun(X)$-modules is faithfully exact. In particular, $f^*$ is injective.
    \end{lem}

    For a superscheme $X$, we define the $n$-th jet scheme $\clJ_nX$ and the arc scheme $\clJ X$ by
    \[
        \clJ_nX(S)=X(S[t]/(t^{n+1})),\quad \clJ X(S)=X(S\dbr{t})
    \]
    for any supercommutative algebra $S$. Note that $\clJ X$ is the projective limit of $\{\clJ_n X\}$.
    
    \begin{lem}\label{lem:smooth surj of jet}
        If $f\colon X\to Y$ is smooth surjective, then so is $\clJ_n f\colon \clJ_n X \to \clJ_n Y$ for any $n\geq 0$. In particular, if $X$, $Y$ are affine superschemes, then $(\clJ_n f)^*\colon\Fun(\clJ_n Y)\to \Fun(\clJ_n X)$ and $(\clJ f)^*\colon\Fun(\clJ Y)\to\Fun(\clJ X)$ are injective.  
    \end{lem}

    \begin{proof}
         By \cref{lem:formally smooth,rmk:surjective}, the former claim follows as in the purely even case. Since a smooth surjective morphism is faithfully flat by definition, $(J_n f)^*$ is injective for any $n\geq 0$ by \cref{lem:faithfully flat}. Furthermore, since $\Fun(\clJ X)$ is the inductive limit of $\{\Fun(\clJ_n X)\}$, the injectivity of $(\clJ f)^*$ follows.
    \end{proof}

    The following lemma follows immediately from the construction of the arc scheme of $\bbA^{n|m}$:
    \begin{lem}\label{lem:int of arc apace on affine}
        If $X$ is an affine superspace $\bbA^{n|m}$ over $\bbC$, then for any nonempty open subset $U$ of $X$, the restriction map $\Fun(\clJ X)\to \Fun(\clJ U)$ is injective.
    \end{lem}

    We return to our setting. To show the injectivity of $w_\kappa$, it is enough to show that $\gr w_\kappa\colon \gr V^\kappa_{N=1}(\frg)\to \gr(M\otimes \Pi^\kappa)$ is injective, where $\gr$ denotes the graded vertex Poisson superalgebras and homomorphism associated with the Li filtration of the underlying vertex algebra and homomorphism, appropriately (\cite{Li05}). Using the natural identifications 
    \[
        \gr V^\kappa_{N=1}(\frg)\simeq \Fun\bigl(\clJ(\frg^*\times\Pi\frg^*)\bigr), 
        \quad
        \gr(M\otimes \Pi^\kappa)\simeq \Fun\bigl(\clJ (T^*U\times\frh^*\times\Pi(T^*U\times \frh^*))\bigr),
    \]
     the map $\gr w_\kappa$ is equal to $\clJ\mu^*$, where $\mu^*\colon \Fun(\frg^*\times\Pi\frg^*)\to\Fun(T^*U\times\frh^*\times\Pi(T^*U\times \frh^*))$ is the homomorphism of Poisson superalgebras induced between their Zhu $C_2$-algebras (\cite{Zhu96}). Let $y^\alpha$ be the coordinate on $U$ defined in \cref{ss:Flag supervariety}, and $\eta_\alpha$ be the fiber coordinate on $T^*U$ such that $\{\eta_\alpha,y^\beta\}=\delta^\alpha_\beta$. Let $\psi^\alpha, \psi_\alpha$ be the parity change coordinate system of $y^\alpha,\eta_\alpha$. Let $u_i$ be the coordinate on $\frh^*$ defined in \cref{ss:Flag supervariety}, and let $\phi_i$ be its parity change.

     \begin{prp}\label{prp:C_2-algebra}
        With the above notation, $\mu^*$ is written as 
        \begin{align*}
            \mu^*(\ol{e}_\alpha)&=(-1)^{\abs{\alpha}}\bigl(\psi_{\alpha}+\sum_{\beta\in\Delta^+\cap(\Delta^++\alpha)}P^\beta_\alpha(y)\psi_\beta\bigr),\\
            \mu^*(\ol{h}_i)&=-\sum_{\beta\in\Delta_+}\beta(h_i)y^\beta\psi_\beta+\phi_i,\\
            \mu^*(\ol{f}_\alpha)&=(-1)^{\abs{\alpha}}\bigl(\sum_{\beta\in\Delta^+}Q^\beta_\alpha(y)\psi_\beta+\sum_{i=1}^l R^i_\alpha(y)\phi_i\bigr)+\kappa(e_\alpha,f_\alpha)\psi^\alpha+\sum_{\beta\in\Delta^+}\psi^\beta A_{\alpha,\beta}(y),
        \end{align*}
        and 
        \[
            \mu^*(x)=\sum_\alpha x^\alpha(y)\eta_\alpha+\sum_ix^i(y)u_i+\sum_{\alpha,\beta}\frac{\pd x^\alpha}{\pd y^\beta}(y)\psi^\beta\psi_\alpha+\sum_{i,\gamma}\frac{\pd x^i}{\pd y^\gamma}(y)\psi^\gamma\phi_i
        \]
        for $x\in\frg\subset\Fun(\frg^*)$. Here, $A_{\alpha,\beta}(y)$ is a polynomial with zero constant term, and $x^\alpha(y)$, $x^i(y)$ are polynomials defined by $\rho_{\frn_-}(x)=\sum_{\alpha}x^\alpha(y)\pd_\alpha+\sum_i x^i(y)u_i$. 
     \end{prp}

     \begin{proof}
        By construction, we have 
        \begin{align*}
            \mu^*(\ol{x})&=(-1)^{\abs{x}}\bigl(\sum_\alpha x^\alpha(y)\psi_\alpha+\sum_ix^i(y)\phi_i\bigr)+A(x;y,\psi^*),
            \\
            \mu^*(x)&=\sum_\alpha x^\alpha(y)\eta_\alpha+\sum_ix^i(y)u_i+\sum_{\alpha,\beta}\frac{\pd x^\alpha}{\pd y^\beta}(y)\psi^\beta\psi_\alpha+\sum_{i,\gamma}\frac{\pd x^i}{\pd y^\gamma}(y)\psi^\gamma\phi_i,
        \end{align*}
        for $x\in\frg$, where $A(x;y,\psi^*)$ is a superpolynomial in the variables $y^\alpha$, $\psi^\alpha$. Note that $\frg^*\times\Pi\frg^*$ and $T^*U\times\frh^*\times \Pi(T^*U\times\frh^*)$ have $\bbZ\Delta\oplus\frac{1}{2}\bbZ$-grading induced by those of $V^\kappa_{N=1}(\frg)$ and $M\otimes\Pi^\kappa$, respectively. Since $\mu^*$ preserves these gradings, we have
        \[
            A(e_\alpha;y,\psi^*)=A(h_i;y,\psi^*)=0,
            \quad
            A(f_\alpha;y,\psi^*)=\sum_\beta \psi^\beta A_{\alpha,\beta}(y)
        \]
        for some polynomial $A_{\alpha,\beta}(y)$ in variables $y$, whose $\bbZ\Delta$-grading is $-\alpha+\beta$. Furthermore, since $\mu^*$ is a homomorphism of Poisson superalgebras, we have
        \[
            (-1)^{\abs{\beta}}\delta_{\beta,\alpha}\kappa(e_\beta,f_\alpha)
            =\{\mu^*(\ol{e_\beta}),\mu^*(\ol{f_\alpha})\}
            =(-1)^{\abs{\beta}}\bigl(A_{\alpha,\beta}(y)+\sum_{\gamma\in\Delta^+\cap(\Delta^++\beta)} P_\beta^\gamma(y)A_{\alpha,\gamma}(y)\bigr).
        \]
        By this equation, we have
        \[
            A(f_\alpha;y,\psi^*)=\kappa(e_\alpha,f_\alpha)\psi^\alpha+\sum_{\beta\in\Delta^+}\psi^\beta A_{\alpha,\beta}(y).
        \]
     \end{proof}

     \begin{rmk}
         The $C_2$-algebra of a SUSY vertex algebra has an odd derivation $\sd$ such that $\sd^2=0$ (\cite{Yan22}). In our case, the operator $\sd$ on $\Fun\bigl(\frg^*\times\Pi\frg^*\bigr)$ is 
         \[
            \sd=\sum_ax_a\frac{\pd}{\pd \ol{x}_a}
         \]
         for a fixed basis $\{x_a\}$ of $\frg$, while the operator $\sd$ on $\Fun\bigl(T^*U\times\frh^*\times\Pi(T^*U\times\frh^*)\bigr)$ is 
         \[
            \sd=\sum_\alpha
            \left(
            (-1)^{\abs{\alpha}}\eta_\alpha\frac{\pd}{\pd \psi_\alpha}
            +\psi^\alpha\frac{\pd}{\pd y^\alpha}
            \right)+\sum_i u_i\frac{\pd}{\pd \phi_i}.
         \]
         However, this structure is not used in this paper.
     \end{rmk}

    \begin{thm}\label{thm:injective}
        If $\kappa$ is nondegenerate, then $w_\kappa$ is injective.
    \end{thm}

    \begin{proof}
        We regard $\mu^*$ as the pullback homomorphism of a morphism $\mu\colon T^*U\times\frh^*\times\Pi(T^*U\times\frh^*)\to\frg^*\times\Pi\frg^*$. We first prove that $\mu$ is smooth at
        \[
        y=\eta=\psi^*=\psi=\phi=0,\ u_i=c_i,
        \]
        where the $c_i$ correspond via $\kappa$ to an element $h=\sum_ic_ih_i\in\frh$ such that $\alpha(h)\neq 0$ for any $\alpha\in\Delta^+$.
        Then, the induced map on the cotangent spaces  at this point is 
        \begin{equation}\label{eq:one-forms 1}
            (d\mu)^*(d\ol{e}_\alpha)=d\psi_\alpha,
            \quad
            (d\mu)^*(d\ol{h}_i)=d\phi_i,
            \quad
            (d\mu)^*(d\ol{f}_\alpha)=\kappa(e_\alpha,f_\alpha)d\psi^\alpha. 
        \end{equation}
        \begin{equation}\label{eq:one-forms 2}
            (d\mu)^*(dx)=\sum_{\alpha} x^\alpha(0)d\eta_\alpha+\sum_{i,\gamma}
            \left(
            c_i\frac{\pd x^i}{\pd y^\gamma}(0)dy^\gamma+x^i(0)du_i
            \right)
        \end{equation}
        By the above formulas, we can decompose $(d\mu)^*$ into two linear maps $\Pi\frg\to\Pi(\frn^*_+\oplus\frn_+\oplus\frh)$ in \cref{eq:one-forms 1} and $\frg\to\frn^*_+\oplus\frn_+\oplus\frh$ in \cref{eq:one-forms 2}. Here, we identify the cotangent spaces on $\frg^*\times\Pi\frg^*$ and $T^*U\times\frh^*\times\Pi(T^*U\times\frh^*)$ with $\frg\oplus\Pi\frg$ and $\frn_+^*\oplus\frn_+\oplus\frh\oplus\Pi(\frn_+^*\oplus\frn_+\oplus\frh)$, respectively. 
        
        The map \cref{eq:one-forms 1} is obviously injective. We prove the injectivity of the map \cref{eq:one-forms 2}. We define a morphism $f\colon N_+\times\frb_-\to\frg$ by
        \[
            f(n,b)\ceq \Ad_nb
        \]
        for $n\in N_+$ and $b\in\frb_-$. Identifying $N_+\times\frb_-\simeq N_+\times\frb_+^*\simeq T^*U\times\frh^*$, $\frg\simeq\frg^*$ via $\kappa$, the homomorphism $f^*$ on functions is written as $f^*(x)=\sum_{\alpha}x^\alpha(y)\eta_\alpha+\sum_ix^i(y)u_i$ for $x\in\frg\subset\Fun(\frg^*)$. Since $(df)^*$ at $(e,h)$, $h=\sum_ic_ih_i$ is equal to the map \cref{eq:one-forms 2},  we have only to show that $f$ is smooth at this point. Since the linear map $df$ on the tangent space at $(n,b)$ is $df(x,y)=\Ad_n(y)+[x,b]$ for $x\in \frn_+$, $y\in \frb$, the smoothness of $f$ at $(e,h)$ is clear.        
        
        By the above argument, we have shown that $\mu$ is smooth at the point. Thus, there exists an affine open neighborhood $V$ of this point in $T^*U\times\frh^*\times\Pi(T^*U\times\frh^*)$ such that $\mu|_V$ is smooth. Since a smooth morphism is open, $W\ceq \mu(V)$ is an open subset of $\frg^*\times\Pi\frg^*$. Then, $\mu|_V\colon V\to W$ is smooth and surjective. By \cref{lem:smooth surj of jet}, $(\clJ(\mu|_V))^*=(\clJ\mu|_{\clJ V})^*\colon \Fun(\clJ W)\to\Fun(\clJ V)$ is injective. Since $\frg^*\times\Pi\frg^*$ is an affine superspace, the restriction map $\Fun(\clJ(\frg^*\times\Pi\frg^*))\to\Fun(\clJ W)$ is injective (\cref{lem:int of arc apace on affine}). Thus, we complete the proof.  
    \end{proof}

    \begin{thm}\label{thm:collapse injectivity}
        If $\kappa=0$, then $w_\kappa$ is not injective.
    \end{thm}

    \begin{proof}
        We fix a nondegenerate supersymmetric invariant form $\kappa_0$, and $\kappa=k\kappa_0$ ($k\in\bbC$). Fix bases $\{x_i\}$ and $\{h_i\}$ of $\frg$ and $\frh$, respectively, and let $\{x^i\}$ and $\{h^i\}$ be the dual bases with respect to $\kappa_0$. Let $\rho$ be the element of $\frh$ corresponding to the Weyl vector via $\kappa_0$. If $k\neq 0$, then the $N=1$ superconformal vector $\tau$ of $V^\kappa_{N=1}(\frg)$ is of the form
        \[
            \tau=\left(\frac{1}{k}\sum_i(-1)^{\abs{x_i}}x_{i,-1}\ol{x}^i_{-1/2}
            +\frac{1}{3k^2}\sum_{i,j}(-1)^{\abs{x_i}}\ol{x}_{i,-1/2}\ol{x}_{j,-1/2}\ol{[x^i,x^j]}_{-1/2}\right)\ket{0}.
        \]
        and the corresponding vector is of the form
        \[
            w_\kappa(\tau)=
            \left(
            \sum_\alpha ((-1)^{\abs{\alpha}}\ol{a}_{\alpha,-1/2}\ol{a}^{*,\alpha}_{-1}
            +a_{\alpha,-1}\ol{a}^{*,\alpha}_{-1/2})-\frac{1}{k}(\sum_ib_{i,-1}\ol{b}^i_{-1/2}-2\ol{\rho}_{-3/2})
            \right)\ket{0}    
        \]
        by \cref{prp:superconformal vect}. We set $\nu\ceq k^2\tau$. Clearly, $\nu$ and $w_\kappa(\nu)$ are well-defined for any $k$. If $k=0$, then
        \[
            k^2\tau|_{k=0}
            =\frac{1}{3}\sum_{i,j}(-1)^{\abs{x_i}}\ol{x}_{i,-1/2}\ol{x}_{j,-1/2}\ol{[x^i,x^j]}_{-1/2}\ket{0}
            =\sum_{\alpha\in\Delta^+}\frac{2}{3\kappa_0(e_\alpha,f_\alpha)^2}\ol{e}_{\alpha,-1/2}\ol{f}_{\alpha,-1/2}\ol{[f_\alpha,e_\alpha]}
            \neq 0,
        \]
        whereas $w_\kappa(\nu)=0$.
    \end{proof}

\section{Semi-infinite parabolic induction of SUSY affine Lie superalgebras}

    In this section, we generalize the free field construction to standard parabolic subalgebras. We then show that, for the standard parabolic subalgebra corresponding to a single simple root, the free field homomorphism factors through the SUSY affine vertex algebra associated with the corresponding rank-one Levi subalgebra. This section is a SUSY analogue of \cite[\S 6]{Fre05}, but the treatment of rank-one Levi factors is more sensitive than the ordinary case.

\subsection{Semi-infinite parabolic induction of SUSY affine Lie superalgebras}
    Let $\frp=\frp_I$ be a standard parabolic subalgebra of $\frg$ of the form $\frp_I=\frh\oplus\bigoplus_{\alpha\in\Delta_I\cup\Delta^-}\frg_\alpha$, where $I$ is a subset of $\{\alpha_1,\ldots,\alpha_l\}$, and $\Delta_I\ceq \Delta\cap\bbZ I$. Then, $\frp$ decomposes as:
    \[
        \frp=\frl\oplus\frn_{I,-},
        \quad \frl=\frl_I=\frh\oplus\bigoplus_{\alpha\in\Delta_I}\frg_\alpha,
        \quad 
        \frn_{I,-}=\bigoplus_{\alpha\in\Delta^-\setminus\Delta_I}\frg_\alpha.
    \]
    Then, $\frn_{I,-}$ is an ideal of $\frp$. Let $P$ (resp.\ $L$, $N_{I,-}$) be subgroups of $G$ such that $\Lie(P)=\frp$ (resp.\ $\Lie(L)=\frl$, $\Lie(N_{I,-})=\frn_{I,-}$). The supervariety $P\backslash G$ has an open dense subset $U_I\ceq P\backslash PN_{I,+}\simeq N_{I,+}$, where $N_{I,+}$ is a subgroup of $G$ corresponding to $\frn_{I,+}=\bigoplus_{\alpha\in\Delta^+\setminus \Delta_I}\frg_\alpha$. In particular, setting $\Delta(\frn_{I,+})\ceq \Delta^+\setminus\Delta_I$, we have $\bbA^{\Delta(\frn_{I,+})}\simeq U_I$ via the exponential map $\frn_{I,+}\to N_{I,+}$. Since the pullback of $N_{I,-}\backslash G\to P\backslash G$ over $U_I$ is isomorphic to $L\times U_I$, the space of left $L$-invariant vector fields on $L\times U_I\subset N_{I,-}\backslash G$ is naturally isomorphic to $\Vect(U_I)\oplus \Fun U_I\otimes\frl$. The right $G$-action on $N_{I,-}\backslash G$ induces a homomorphism of Lie superalgebras:
    \begin{equation}\label{eq:realization on grassmanian}
        \rho_{\frn_{I,-}}\colon\frg\lto\Vect(U_I)\oplus \Fun U_I\otimes\frl
    \end{equation}

    The following claim is the superaffinization of \cref{eq:realization on grassmanian}. The proof is similar to that of \cref{thm:free field realizations}.
    \begin{thm}\label{thm:semi-infinite parabolic induction}
        Let $M_{\Delta(\frn_{I,+})}$ be the $\beta\gamma$--$bc$ system of $\Delta(\frn_{I,+})$. There exists a homomorphism of SUSY vertex algebras:
        \[
            w^\frp_\kappa\colon V^\kappa_{N=1}(\frg)\lto M_{\Delta(\frn_{I,+})}\otimes V^{\kappa}_{N=1}(\frl).
        \]
        Here, $\kappa$ on the right-hand side denotes the restriction of $\kappa$ to $\frl$ by abuse of notation.
    \end{thm}

    \begin{proof}
        By \cref{eq:realization on grassmanian}, we have a homomorphism of Lie superalgebras
        \begin{equation}\label{eq:superloop realization on grassmaniann}
            \clL^{N=1}\rho_{\frn_{I,-}}\colon\clL^{N=1}\frg\lto\Vect(\clL^{N=1}U_I)\oplus\Fun(\clL^{N=1}U_I)\wh{\otimes}\clL^{N=1}\frl.
        \end{equation}
        Let $\wt{\clA}^{\Delta(\frn_{I,+})}$ be a completed SUSY Weyl algebra associated with $\Delta(\frn_{I,+})$, which is isomorphic to $U(M_{\Delta(\frn_{I,+})})$. Let $\clA^{\Delta(\frn_{I,+})}_{\loc,0}$ (resp.\ $\clA^{\Delta(\frn_{I,+})}_{\loc,\leq 1}$) be the intersection of $\Lie(M_{\Delta(\frn_{I,+})})$ and $\wt{\clA}_0^{\Delta(\frn_{I,+})}$ (resp.\ $\wt{\clA}_{\leq 1}^{\Delta(\frn_{I,+})}$), and let $\clT^{\Delta(\frn_{I,+})}_{\loc}$ be the image of $\clA^{\Delta(\frn_{I,+})}_{\loc,\leq 1}\to\Vect(\clL^{N=1}U_I)$. By construction, the homomorphism \cref{eq:superloop realization on grassmaniann} factors through $\clT^{\Delta(\frn_{I,+})}_\loc\oplus \clA^{\Delta(\frn_{I,+})}_{\loc,0}\wh{\otimes}\clL^{N=1}\frl$.
        
        By definition, there exists an exact sequence 
        \[
            \begin{tikzcd}
                0\ar[r] &\clA^{\Delta(\frn_{I,+})}_{\loc,0} \ar[r] 
                &\clA^{\Delta(\frn_{I,+})}_{\loc,\leq 1}\oplus \clA^{\Delta(\frn_{I,+})}_{\loc,0}\wh{\otimes}\frl^\kappa_{N=1} \ar[r]
                &\clT^{\Delta(\frn_{I,+})}_\loc\oplus \clA^{\Delta(\frn_{I,+})}_{\loc,0}\wh{\otimes}\clL^{N=1}\frl \ar[r]
                & 0.
            \end{tikzcd}
        \]
        The problem is to find a homomorphism of Lie superalgebras $\frg^\kappa_{N=1}\to\clA^{\Delta(\frn_{I,+})}_{\loc,\leq 1}\oplus \clA^{\Delta(\frn_{I,+})}_{\loc,0}\wh{\otimes}\frl^\kappa_{N=1}$ lifting \cref{eq:superloop realization on grassmaniann}. By taking normally ordered products, we have a linear map $\wt{w}^\frp_\kappa\colon\clL^{N=1}\frg\to \clA^{\Delta(\frn_{I,+})}_{\loc,\leq 1}\oplus \clA^{\Delta(\frn_{I,+})}_{\loc,0}\wh{\otimes}\frl^\kappa$. Then, we have an even local $2$-cocycle $\omega^\frp_\kappa$ with coefficients in $\clA^{\Delta(\frn_{I,+})}_{\loc,0}$ defined by
        \[
            \omega^\frp_\kappa(\ol{x}(Z),\ol{y}(W))\ceq [\wt{w}^\frp_\kappa(\ol{x}(Z)),\wt{w}^\frp_\kappa(\ol{y}(W))]-\wt{w}^\frp_\kappa([\ol{x}(Z),\ol{y}(W)]).
        \]
        It is enough to show that $\omega^\frp_\kappa$ and $\sigma_\kappa$ are equivalent in $H^2(\clL^{N=1}\frg,\clA^{\Delta(\frn_{I,+})}_{\loc,0})$. To prove this, it suffices to compare their restrictions to $\clL^{N=1}\frl$. Indeed, the argument of \cref{s:Comparison of cohomology classes} applies verbatim with $\frb_-$ and $\frh$ replaced by $\frp_I$ and $\frl_I$, respectively.

        Let $e_\alpha$ be root vectors in $\frn_{I,+}$ for $\alpha\in\Delta(\frn_{I,+})$. For $x\in\frl$, we define constants $c_\alpha^\beta$ by $[x,e_\alpha]=\sum_\beta c^\beta_\alpha e_\beta$ which makes sense since $\frn_{I,+}$ is stable under the adjoint action of $\frl$. By construction, we can write
        \[
            \wt{w}^\frp_\kappa(\ol{x}(Z))\ceq -\sum_{\alpha,\beta\in\Delta(\frn_I^+)}c_\alpha^\beta \nop{a^{*,\alpha}(Z)\ol{a}_\alpha(Z)}+\ol{x}(Z),
        \]
        where $\ol{x}(Z)$ on the right-hand side denotes the superfield corresponding to $\ol{x}\in\Pi\frl$ on $V^\kappa_{N=1}(\frl)$. Thus, we have
        \[
            \omega^\frp_\kappa(\ol{x}(Z),\ol{y}(W))=\pd_W^{(0|1)}\delta(Z,W)\kappa(x,y)=\sigma_\kappa(\ol{x}(Z),\ol{y}(W))
        \]
        for $x,y\in\frl$.
    \end{proof}

\subsection{Factorizations of free field realizations}
    For any simple root $\alpha_i$, we put $\Delta_i\ceq \Delta\cap\bbZ\alpha_i$. We write the corresponding parabolic subalgebra $\frp_i\ceq\frp_{\Delta_i}$, which decomposes as follows:
    \[
        \frp_i=\frl_i\oplus\frn_{i,-},
        \quad
        \frl_i=\frh\oplus\sum_{\alpha\in\Delta_i}\frg_\alpha,
        \quad
        \frn_{i,-}=\bigoplus_{\alpha\in\Delta^-\setminus\Delta_i}\frg_{\alpha}
    \]
    As a vector superspace, $\frl_i$ is of the form:
    \begin{itemize}
        \item 
        $\alpha_i$ is even: $\frl_i=\frsl_2\oplus\bigoplus_{j\neq i}\bbC h_j$,
        \item 
        $\alpha_i$ is odd non-isotropic: $\frl_i=\frosp_{1|2}\oplus\bigoplus_{j\neq i}\bbC h_j$,
        \item 
        $\alpha_i$ is odd isotropic: $\frl_i=\frsl_{1|1}\oplus\bigoplus_{j\neq i}\bbC h_j$. 
    \end{itemize}
    In each case, $e,f,h_i$ denote the elements of $\frl_i$ corresponding to the standard Chevalley generators of $\frsl_{2}$, $\frosp_{1|2}$ or $\frsl_{1|1}$, as appropriate. Note that $\alpha_i(h_i)=2,1,0$, respectively.
    
    \begin{lem}\label{lem:realization of triple}
        Let $M_i$ be a $\beta\gamma$--$bc$ system generated by 
        \begin{itemize}
            \item
            $\alpha_i$ is even: $a^*(Z),\ol{a}(Z)$, $a$ is even,
            \item 
            $\alpha_i$ is odd non-isotropic: $a^*(Z),\phi^*(Z),\ol{a}(Z),\ol{\phi}(Z)$, $a$ is even, $\phi$ is odd,
            \item
            $\alpha_i$ is odd isotropic: $\phi^*(Z), \ol{\phi}(Z)$, $\phi$ is odd.
        \end{itemize}
        The following linear map $w^i_\kappa\colon V^\kappa_{N=1}(\frl_i)\to M_i\otimes\Pi^\kappa$ is a homomorphism of SUSY vertex algebras:
        \begin{itemize}
            \item 
            $\alpha_i$ is even:
            \begin{align*}
                w^i_\kappa(\ol{e}(Z))&=\ol{a}(Z),
                \quad 
                \\
                w^i_\kappa(\ol{f}(Z))&=-\nop{a^*(Z)^2\ol{a}(Z)}-\kappa(e,f)\pd_Z^{(0|1)}a^*(Z)+a^*(Z)\ol{b}_i(Z),
                \\
                w^i_\kappa(\ol{h}_j(Z))&= -\alpha_i(h_j)\nop{a^*(Z)\ol{a}(Z)}+\ol{b}_j(Z)
            \end{align*}
            \item
            $\alpha_i$ is odd non-isotropic: 
            \begin{align*}
            w^i_\kappa(\ol{e}(Z))&=\ol{\phi}(Z)+\nop{\phi^*(Z)\ol{a}(Z)},
            \\
            w^i_\kappa(\ol{f}(Z))&=\nop{a^*(Z)\ol{\phi}(Z)}+\nop{\phi^*(Z)a^*(Z)\ol{a}(Z)}-\kappa(e,f)\pd_Z^{(0|1)}\phi^*(Z)-\phi^*(Z)\ol{b}_i(Z),
            \\
            w^i_\kappa(\ol{h}_j(Z))&= -2\alpha_i(h_j)\nop{a^*(Z)\ol{a}(Z)}+\alpha_i(h_j)\nop{\phi^*(Z)\ol{\phi}(Z)}+\ol{b}_j(Z)
            \end{align*}
            \item
            $\alpha_i$ is odd isotropic:
            \begin{align*}
                w^i_\kappa(\ol{e}(Z))&=\ol{\phi}(Z),\\
                w^i_\kappa(\ol{f}(Z))&=-\kappa(e,f)\pd_Z^{(0|1)}\phi^*(Z)-\phi^*(Z)\ol{b}_i(Z),
                \\
                w^i_\kappa(\ol{h}_j(Z))&= \alpha_i(h_j)\nop{\phi^*(Z)\ol{\phi}(Z)}+\ol{b}_j(Z).
            \end{align*}
        \end{itemize} 
    \end{lem}   

    \begin{rmk}
        Although $\mathfrak{sl}_{1|1}$ admits a nondegenerate supersymmetric invariant bilinear form, it is not simple, and hence our general free field construction does not apply directly. Nevertheless, a direct computation shows that the construction above still gives the desired free field realization in this case.
    \end{rmk}
    
    \begin{cor}\label{cor:identification of Wakimoto modules}
        Identifying $M_i\otimes M_{\Delta^+\setminus\Delta_i}$ with $M_{\Delta^+}$ via the isomorphism induced by the multiplication map, we have $w_\kappa=(1\otimes w^i_\kappa)\circ w^{\frp_i}_\kappa$.
    \end{cor}

\section{Screening operators}
    Using the rank-one factorizations obtained in the previous section, we construct screening operators associated with the simple roots. We then prove that these operators commute with the action of the SUSY affine Lie superalgebra. This section corresponds to \cite[\S~7.3 and 8.2]{Fre05}, but the situation is more complicated.

\subsection{The exponential superfields} 
        
        In this section, we assume that $\kappa$ is nondegenerate.
        
        For $\lambda\in\frh^*$, we introduce a superfield $e^\lambda(Z)\colon\Pi^{\kappa}_\mu\to z^{\kappa^{-1}(\lambda,\mu)}\Pi^{\kappa}_{\lambda+\mu}\dpr{Z}$ by
        \[
            e^\lambda(Z)\ceq
            T_\lambda z^{\kappa^{-1}(\lambda,\mu)}
            \exp\Bigl(-\sum_{n\leq -1}\frac{z^{-n}}{n}\lambda_n+\sum_{n\leq -1/2}\theta z^{-n-1/2}\ol{\lambda}_n\Bigr)
            \exp\Bigl(-\sum_{n\geq 1}\frac{z^{-n}}{n}\lambda_n+\sum_{n\geq 1/2}\theta z^{-n-1/2}\ol{\lambda}_n\Bigr).
        \]
        Here, 
        \begin{itemize}
            \item 
            $T_\lambda$ is a unique linear map $\Pi^{\kappa}_\mu\to\Pi^{\kappa}_{\lambda+\mu}$ such that 
            \[
            T_\lambda\ket{\mu}=\ket{\lambda+\mu}, 
            \quad
            [b_{i,n},T_\lambda]=\lambda(h_i)\delta_{n,0}, 
            \quad
            [\ol{b}_{i,n},T_\lambda]=0,
            \]
            \item
            $\lambda_n\ceq t^n\lambda, \ol{\lambda}_{n+1/2}\ceq t^n\omega\lambda\in\wh{\frh}^\kappa_{N=1}$, where $\lambda$ is regarded as an element of $\frh$ via $\kappa$,
            \item 
            $\kappa^{-1}$ denotes the nondegenerate supersymmetric bilinear form on $\frh^*$ corresponding to $\kappa$.
        \end{itemize}

        \begin{lem}
            For any $\lambda,\mu\in\frh^*$, and $h_i\in\frh$, we have
            \begin{align*}
                \sd_Ze^\lambda(Z)&=\nop{\ol{\lambda}(Z)e^\lambda(Z)}\\
                e^\lambda(Z)e^\mu(W)&=(Z-W)^{\kappa^{-1}(\lambda,\mu)|0}\nop{e^\lambda(Z)e^\mu(W)},\\
                \ol{b}_i(Z)e^\lambda(W)&=(Z-W)^{-1|1}
                \lambda(h_i)e^\lambda(W)+\nop{\ol{b}_i(Z)e^\lambda(W)},
            \end{align*}
            where $\kappa^{-1}$ is a bilinear form on $\frh^*$ corresponding to $\kappa$ on $\frh$.
        \end{lem}

        \begin{dfn}\label{dfn:screening operator}
            For $i=1,\ldots,l$, we define the superfield $S_i(Z)$ by
            \[
                S_i(Z)\ceq \ol{e}_i^L(Z)e^{-\alpha_i}(Z)\colon W^\kappa_0\to W^{\kappa}_{-\alpha_i}\dpr{Z},
            \]
            and $S_i\ceq \oint dZ\ S_i(Z)$. We call $S_i(Z)$ the \emph{screening current} and $S_i$ the \emph{screening operator}. Here, a superfield $\ol{e}_i^L(Z)$ is defined by $\ol{e}_i^L(Z)\ceq-\ol{a}_{\alpha_i}(Z)+\sum_\beta (-1)^{\abs{P^{L,\beta}_i}}\nop{P^{L,\beta}_i(a^*(Z))\ol{a_\beta}(Z)}$, where $P^{L,\beta}_i$ is the polynomial in \cref{left action of Lie algebra}.
        \end{dfn}

    \begin{thm}\label{thm:screening operator}
        Let $\frg$ be a basic classical Lie superalgebra with a nondegenerate invariant supersymmetric form $\kappa$. For any $i=1,\ldots, l$, the screening operator $S_i$ commutes with the $\wh{\frg}_{N=1}^\kappa$-action. 
    \end{thm}

    \begin{proof}
        We fix $i$. Identifying $M_{\Delta^+}\otimes\Pi^\kappa_0\simeq M_i\otimes\Pi^\kappa_{0}\otimes M_{\Delta^+\setminus \Delta_i}$, we see that $S_i(Z)$ is of the form
        \begin{itemize}
            \item 
            $\alpha_i$ is even: $S_i(Z)=-\ol{a}(Z)e^{-\alpha_i}(Z)$,
            \item 
            $\alpha_i$ is odd non-isotropic: $S_i(Z)=(-\ol{\phi}(Z)+\nop{\phi^*(Z)\ol{a}(Z)})e^{-\alpha_i}(Z)$,
            \item 
            $\alpha_i$ is odd isotropic: $S_i(Z)=-\ol{\phi}(Z)e^{-\alpha_i}(Z)$.
        \end{itemize}
        Thus, it is enough to show that $S_i$ commutes with the $\wh{\frl_i}^\kappa_{N=1}$-action, which can be checked directly using \cref{lem:realization of triple}. Since commutativity with $e$ and $h_j$ is immediate, it remains only to check commutativity with $f$. Note that $\alpha_i=h_i/\kappa(e_i,f_i)$.
        \begin{itemize}
            \item 
            $\alpha_i$ is even:
            \begin{align*}
                \ol{f}(Z)S_i(W) 
                &=(\nop{a^*(Z)^2\ol{a}(Z)}+\kappa(e,f)\pd_Z^{(0|1)}a^*(Z)-a^*(Z)\ol{b}_i(Z))\ol{a}(W)e^{-\alpha_i}(W)\\
                &\sim (-2(Z-W)^{-1|1}\nop{a^*(W)\ol{a}(W)}
                -\kappa(e,f)(Z-W)^{-1|0})e^{-\alpha_i}(W)\\
                &\quad +(Z-W)^{-1|1}\nop{\ol{b}_i(W)e^{-\alpha_i}(W)}
                 +2(Z-W)^{-1|1}\nop{a^*(W)\ol{a}(W)}e^{-\alpha_i}(W)\\
                &=
                \kappa(e,f)
                \left(
                -(Z-W)^{-1|0}e^{-\alpha_i}(W)+(Z-W)^{-1|1}\nop{\frac{\ol{b_i}(W)}{\kappa(e,f)}e^{-\alpha_i}(W)}
                \right)\\
                &=\kappa(e,f)\sd_w(Z-W)^{-1|1}e^{-\alpha_i}(W).
            \end{align*}
            \item 
            $\alpha_i$ is odd non-isotropic
            \begin{align*}
                \ol{f}(Z)S_i(W)
                &=(\nop{a^*(Z)\ol{\phi}(Z)}+\nop{\phi^*(Z)a^*(Z)\ol{a}(Z)}-\kappa(e,f)\pd_Z^{(0|1)}\phi^*(Z)\\
                &\quad
                -\phi^*(Z)\ol{b}_i(Z))(-\ol{\phi}(W)+\nop{\phi^*(W)\ol{a}(W)})e^{-\alpha_i}(W)\\
                &\sim((Z-W)^{-1|1}\nop{a^*(W)\ol{a}(W)}
                -(Z-W)^{-1|1}\nop{\ol{\phi}(W)\phi^*(W)}\\
                &\quad
                -(Z-W)^{-1|1}\nop{a^*(W)\ol{a}(W)}
                -\kappa(e,f)(Z-W)^{-1|0})e^{-\alpha_i}(W)\\
                &\quad
                +(Z-W)^{-1|1}\nop{\ol{b}_i(W)e^{-\alpha_i}(W)}
                +(Z-W)^{-1|1}\nop{\phi^*(W)\ol{\phi}(W)}e^{-\alpha_i}(W)\\
                &=
                \kappa(e,f)\left(
                -(Z-W)^{-1|0}e^{-\alpha_i}(W)+(Z-W)^{-1|1}\nop{\frac{\ol{b}_i(W)}{\kappa(e,f)}e^{-\alpha_i}(W)}
                \right)\\
                &=
                \kappa(e,f)\sd_w(Z-W)^{-1|1}e^{-\alpha_i}(W).
            \end{align*}
            \item
            $\alpha_i$ is odd isotropic:
            \begin{align*}
                \ol{f}(Z)S_i(W)
                &=(\kappa(e,f)\pd_Z^{(0|1)}\phi^*(Z)+\phi^*(Z)\ol{b}_i(Z))\ol{\phi}(W)e^{-\alpha_i}(W)\\
                &\sim-\kappa(e,f)(Z-W)^{-1|0}e^{-\alpha_i}(W)+(Z-W)^{-1|1}\nop{\ol{b}_i(W)e^{-\alpha_i}(W)}\\
                &=\kappa(e,f)\left(
                -(Z-W)^{-1|0}e^{-\alpha_i}(W)+(Z-W)^{-1|1}\nop{\frac{\ol{b}_i(W)}{\kappa(e,f)}e^{-\alpha_i}(W)}
                \right)\\
                &=\kappa(e,f)\sd_w(Z-W)^{-1|1}e^{-\alpha_i}(W).
            \end{align*}
            \end{itemize}
            Thus, we have the claim.
    \end{proof}

    \begin{cor}\label{cor:screening operator}
        If $\kappa$ is nondegenerate, identifying $V^\kappa_{N=1}(\frg)$ with its image, we have $V^\kappa_{N=1}(\frg)\subset\bigcap_{i=1}^l\ker(S_i)$.
    \end{cor}

\section*{Acknowledgments}
    The author would like to thank his advisor, S.\ Yanagida, for his continued guidance and encouragement. The author is grateful to N.\ Genra for suggesting the use of the base affine space in this work, for sharing his expertise on screening operators, and for many helpful discussions. The author used ChatGPT (OpenAI) for language editing and for drafting a small amount of explanatory text. The author takes full responsibility for the content of the manuscript. This work was supported by JSPS KAKENHI Grant Number JP25KJ1393.

\appendix

\section{The mode algebras of SUSY vertex algebras}

In this appendix, we construct the mode algebra of a SUSY vertex algebra, following the corresponding construction for ordinary vertex algebras in \cite{FBZ04}. As pointed out by Frenkel in \cite{Fre07}, however, the completion used in \cite{FBZ04} needs to be modified. We therefore follow the corrected construction given in \cite{Fre07}.

\subsection{Definition of the mode algebras}
    \begin{dfn}[{\cite{HK07}}]
        For a SUSY vertex algebra $V$, we define a Lie superalgebra $\Lie(V)$ by
        \[
            \Lie(V)\ceq \Pi V\dpr{Z}/\Img(S\otimes 1+1\otimes\sd_z))  
        \]
        with a Lie bracket given by
        \[
            [\int a\otimes f(Z),\int b\otimes g(Z)]\ceq(-1)^{\abs{f}\abs{b}}\int\oint dW Y(a,W)bf(Z+W)g(Z),
        \]
        where $\int$ denotes the composition of the quotient map $V\dpr{Z}\to \Lie(V)$ and the parity change map, and $Z+W\ceq (z+w+\theta\zeta,\theta+\zeta)$.
    \end{dfn}

    \begin{prp}
        Let $V$ be a SUSY vertex algebra. We define a linear map $\Lie(V)\to \End(V)$ by $\int a\otimes f(Z)\mto \oint dZ Y(a,Z)f(Z)$. Then, this map is a homomorphism of Lie superalgebras.
    \end{prp}

    For any $a\in V$, $j\in\bbZ$, $J=0,1$, we set $a_{[j|J]}\ceq (-1)^{\abs{a}J}\int a\otimes Z^{j|J}$. Then, the above homomorphism $\Lie(V)\to \End(V)$ is written as $a_{[j|J]}\lmto a_{(j|J)}$. For any $a\in V$, we define a formal superseries $Y[a,Z]$ with coefficients in $\Lie(V)$ by 
    \[
        Y[a,Z]\ceq\sum_{j\in\bbZ, J=0,1}Z^{-j-1|1-J}a_{[j|J]}.
    \]

    Assume that $V$ is a graded SUSY vertex algebra. For a homogeneous element $a\in V$, $\Delta_a$ denotes its grading. Let $I_M$ be a left ideal of $U(\Lie(V))$ generated by $a_{[j|J]}$ for $j-\Delta_a>M$ and $J=0,1$. Since the multiplication of $U(\Lie(V))$ is continuous, it extends to its completion $\wt{U}(\Lie(V))$, and hence $\wt{U}(\Lie(V))$ is a topological algebra.   
    
    \begin{dfn}
    We define an algebra $U(V)$ to be a quotient of $\wt{U}(\Lie(V))$ by the two-sided ideal generated by
    \[
        \nop{Y[a,Z]Y[b,Z]}-Y[a_{(-1|1)}b,Z]
    \]
    for $a,b\in V$. We call $U(V)$ the \emph{mode algebra} of $V$.
    \end{dfn}

\subsection{The case of $V^\kappa_{N=1}(\frg)$}
    In this subsection, $\frg$ is a Lie superalgebra with an invariant supersymmetric form $\kappa$, and we use the following generators:  
    \[
        \ol{x}_{(j|J)}\ceq
        \left\{
        \begin{array}{cc}
            \ol{x}_{j+1/2} & (J=1),\\
            x_{j} & (J=0)
        \end{array}
        \right.
    \]
    for $x\in\frg$, $j\in\bbZ$, $J=0,1$.
    \begin{dfn}
        Let $U(\wh{\frg}^{\kappa}_{N=1})$ be the quotient algebra of the universal enveloping algebra of $\wh{\frg}_{N=1}^\kappa$ by the two-sided ideal generated by $K-1$. Let $I_M$ be a left ideal generated by $\ol{x}_{(j|J)}$ for $x\in\frg$, $j\geq M$. Then, $U(\wh{\frg}^\kappa_{N=1})$ has a topology for which $\{I_M\}_{M\geq 0}$ is a basis of open neighborhoods of $0$. Since the multiplication of $U(\wh{\frg}^\kappa_{N=1})$ is continuous, it extends to its completion $\wt{U}(\wh{\frg}^\kappa_{N=1})$, and hence $\wt{U}(\wh{\frg}^\kappa_{N=1})$ is a topological algebra.   
    \end{dfn}

    \begin{lem}
        There exists a homomorphism of algebras $U(V^\kappa_{N=1}(\frg))\to\wt{U}(\wh{\frg}^\kappa_{N=1})$.
    \end{lem}

    \begin{proof}
        Let $A(Z)$ and $B(Z)$ be formal superseries with coefficients in $\wt{U}(\wh{\frg}^\kappa_{N=1})$ such that for any $M>0$, their images $A(Z)$, $B(Z)$ in $U(\wh{\frg}^\kappa_{N=1})/I_M$ form Laurent superseries. Then, $\nop{A(Z)B(Z)}$ is a well-defined superseries and satisfies the same property as $A(Z)$ and $B(Z)$. 
        \[
            \nop{A(Z)B(Z)}=A(Z)_+B(Z)+(-1)^{\abs{A}\abs{B}}B(Z)A(Z)_-
        \]
        Since $A(Z)_-$ has only finitely many nonzero terms in $U(\wh{\frg}^\kappa_{N=1})/I_M$, there exists $k_0>0$ such that $\ol{x}_{(k|K)}A(Z)_-=0$ in $U(\wh{\frg}^\kappa_{N=1})/I_M$ for $k\geq k_0$. In other words, $A(Z)_-$ is annihilated by the left action of the ideal $I_{k_0}$ in $U(\wh{\frg}^\kappa_{N=1})/I_M$. Since $B(Z)$ is a Laurent superseries in $U(\wh{\frg}^\kappa_{N=1})/I_{k_0}$, for $j\gg 0$, $B_{(j|J)}$ is contained in $I_{k_0}$. Thus, we see that $B(Z)A(Z)_-$ is a Laurent superseries in $U(\wh{\frg}^\kappa_{N=1})/I_M$. Obviously, $A(Z)_+B(Z)$ is a Laurent superseries in $U(\wh{\frg}^\kappa_{N=1})/I_M$. 
        
        By the above claim, we can define a linear map $\Lie(V)\to\wt{U}(\wh{\frg}^\kappa_{N=1})$ by
        \[
            \int\ol{x}^1_{(-j_1-1|1-J_1)}\cdots\ol{x}^r_{(-j_r-1|1-J_r)}\ket{0}\otimes f(Z)\lmto \oint dZ\nop{(\pd_Z^{(j_1|J_1)}\ol{x}^1(Z))\cdots(\pd_Z^{(j_r|J_r)}\ol{x}^r(Z))}f(Z)
        \]
        for $x^i\in\frg$, $j_i\in\bbZ_{\geq 0}$, $J_i=0,1$, and $f(Z)\in\bbC\dpr{Z}$. We can show that it is a homomorphism of Lie superalgebras in the same way as \cite[Proposition~4.2.2]{FBZ04}. By construction, this map induces the homomorphism of algebras $U(V^\kappa_{N=1}(\frg))\to\wt{U}(\wh{\frg}^\kappa_{N=1})$.
    \end{proof}
    
    \begin{prp}
        The above homomorphism $U(V^\kappa_{N=1}(\frg))\to \wt{U}(\wh{\frg}_{N=1})$ is an isomorphism.
    \end{prp}

    \begin{proof}
        We have only to construct the inverse map $\wt{U}(\wh{\frg}^\kappa_{N=1})\to U(V^\kappa_{N=1}(\frg))$. Since the map $\wh{\frg}^\kappa_{N=1}\to U(V^\kappa_{N=1}(\frg))$ defined by $\ol{x}(Z)\mto Y[\ol{x}_{(-1|1)}\ket{0},Z]$ is a homomorphism of Lie superalgebras, we have a homomorphism of algebras $U(\wh{\frg}^\kappa_{N=1})\to U(\Lie(V^\kappa_{N=1}(\frg)))$. Since this homomorphism is obviously continuous, we obtain a homomorphism of $\wt{U}(\wh{\frg}^\kappa_{N=1})\to \wt{U}(\Lie(V^\kappa_{N=1}(\frg)))$, and hence $\wt{U}(\wh{\frg}^\kappa_{N=1})\to U(V^\kappa_{N=1}(\frg))$. We can check that it gives the inverse.
    \end{proof}

\subsection{The case of $\beta\gamma$--$bc$ system}
    In this subsection, let $M$ be a $\beta\gamma$--$bc$ system generated by $a^*(Z)$ and $\ol{a}(Z)$. Let $\clA$ be the corresponding SUSY Weyl algebra and $\wt{\clA}$ its completion. We use the following generators:    
    \[
        a^*_{(j|J)}\ceq
        \left\{
        \begin{array}{cc}
            a^*_{j+1} & (J=1),\\
            \ol{a}^*_{j+1/2} & (J=0),
        \end{array}
        \right.
        \quad
        \ol{a}_{(j|J)}\ceq
        \left\{
        \begin{array}{cc}
            \ol{a}_{j+1/2} & (J=1),\\
            a_j & (J=0).
        \end{array}
        \right.
    \]

    The following statement can be proved in the same way as in the SUSY affine case:
    \begin{prp}\label{prp:mode alg of beta gamma bc sys}
        There exists an isomorphism of algebras $U(M)\sto \wt{\clA}$ such that $\int P\otimes f(Z)\to \oint dZ\nop{P(Z)}f(Z)$ for any polynomial $P$ in $a^*_{(j|J)}$ and $\ol{a}_{(j|J)}$ ($j<0$, $J=0,1$).
    \end{prp}

    \begin{prp}\label{prp:emb of current LSA}
        The natural map $\Lie(M)\to U(M)$ is injective.
    \end{prp}

    \begin{proof}
        Assume that $\oint dZ \nop{P(Z)}f(Z)=0$ and $P$ is a homogeneous polynomial in $a^*_{(j|J)}$, $\ol{a}_{(j|J})$ of degree $d\geq 0$. If $d=0$, then $\int P\otimes f(Z)=0$ immediately follows from assumption. We consider the case of $d>0$. Using the following OPEs: 
        \begin{align*}
            [\ol{a}(Z),\nop{P(W)}f(W)]&=\sum_{j\geq 0;J=0,1}(\pd_W^{(j|J)}\delta(Z,W))(-1)^J\nop{\frac{\pd P}{\pd a^*_{(-j-1|1-J)}}(W)}f(W),\\
            [a^*(Z),\nop{P(W)}f(W)]&=\sum_{j\geq 0;J=0,1}(\pd_W^{(j|J)}\delta(Z,W))(-1)^J\nop{\frac{\pd P}{\pd \ol{a}_{(-1-j|1-J)}}(W)}f(W),
        \end{align*}
        we have 
        \[
            \sum_{j\geq 0;J=0,1}(-\pd_W)^{(j|J)}\left(\nop{\frac{\pd P}{\pd a^*_{(-j-1|1-J)}}(W)}f(W)\right)
            =\sum_{j\geq 0;J=0,1}(-\pd_W)^{(j|J)}\left(\nop{\frac{\pd P}{\pd \ol{a}_{(-j-1|1-J)}}(W)}f(W)\right)
            =0
        \]
        as formal superseries with coefficients in $\wt{\clA}$. Since the state-field correspondence is injective, we have 
        \[
            \sum_{j\geq 0;J=0,1}(-\pd)^{(j|J)}\left(\frac{\pd P}{\pd a^*_{(-1-j|1-J)}}f(Z)\right)
            =\sum_{j\geq 0;J=0,1}(-\pd)^{(j|J)}\left(\frac{\pd P}{\pd \ol{a}_{(-j-1|1-J)}}f(Z)\right)
            =0,      
        \]
        where $\pd$ denotes $(T+\pd_z, S+\sd_z)$. Since we can write 
        \[
            P=\frac{1}{d}\sum_{j\geq 0;J=0,1}\left(a^*_{(-1-j|1-J)}\frac{\pd P}{\pd a^*_{(-j-1|1-J)}}+\ol{a}_{(-j-1|1-J)}\frac{\pd P}{\pd \ol{a}_{(-j-1|1-J)}}\right),
        \]
        we have
        \begin{align*}
            &\int P\otimes f(Z) \\
            &=\frac{1}{d}\int \left(a^*_{(-1|1)}\sum_{j\geq 0, J}(-\pd)^{(j|J)}\left(\frac{\pd P}{\pd a^*_{(-1-j|1-J)}} f(Z)\right)
            +\ol{a}_{(-1|1)}\sum_{j\geq 0,J}(-\pd)^{(j|J)}\left(\frac{\pd P}{\pd \ol{a}_{(-j-1|1-J)}} f(Z)\right)\right)\\
            &=0
        \end{align*}
        in $\Lie(M)$. 
    \end{proof}

\subsection{A geometric interpretation of the completed SUSY Weyl algebra}\label{ss:geom interpretation}
    The aim of this subsection is to prove \cref{prp:geom interpretation of completed Weyl alg}. For simplicity, we assume that $U=\bbA^{1|0}$ or $U=\bbA^{0|1}$, and we write the coordinates on $\clL^{N=1}U$ and the corresponding derivations as
    \[
        y_n\ceq 
        \left\{
        \begin{array}{cc}
            y_n & (n\in\bbZ),\\
            \eta_n & (n\in\bbZ+1/2)
        \end{array}
        \right.
        \quad
        \pd_n\ceq 
        \left\{
        \begin{array}{cc}
            \dfrac{\pd}{\pd y_n} & (n\in\bbZ),\\
            \dfrac{\pd}{\pd \eta_n} & (n\in\bbZ+1/2)
        \end{array}
        \right.
    \]
    for $n\in\frac{1}{2}\bbZ$. We also write generators of $\clA$ as
    \[
        a^*_n\ceq 
        \left\{
        \begin{array}{cc}
            a^*_n & (n\in\bbZ),\\
            \ol{a}^*_n & (n\in\bbZ+1/2)
        \end{array}
        \right.
        \quad
        a_n\ceq
        \left\{
        \begin{array}{cc}
            a_n & (n\in\bbZ),\\
            \ol{a}_n & (n\in\bbZ+1/2)
        \end{array}
        \right.
    \]
    for $n\in\frac{1}{2}\bbZ$. 

    Recall that a continuous vector field $\theta$ on $\clL^{N=1}U$ is an operator on $\Fun(\clL^{N=1}U)$ satisfying the Leibniz rule and the condition that for any $N\geq 0$, $\theta(I_M)\subset I_N$ for $M\ll 0$, where $I_M$ is the kernel of $\Fun(\clL^{N=1}U)\to\Fun(\clL^{N=1}_MU)$. Such an operator $\theta$ is of the form
    \[
        \theta=\sum_{n\geq 0}P_n(y)\pd_n+\sum_{n<0}y_nQ_n(y,\pd)+R(y,\pd),
    \]
    where $P_n(y)$ is a polynomial in $y_n$, and $Q_n(y,\pd)$ and $R(y,\pd)$ are vector fields involving only finitely many $y_n$ and $\pd_n$. 

    \begin{dfn}
    Let $\clA^\natural$ be a completion with respect to a topology whose basis of open neighborhoods of $0$ consists of the subspaces $\sum_{n>N}\clA a_n+\sum_{m>M}a^*_m\clA$. Then, $\clA^\natural$ has a topological algebra structure whose multiplication is induced by that of $\clA$. 
    
    Let $\clA^\natural_0$ be a supercommutative subalgebra of $\clA^\natural$ topologically generated by $a^*_n$ ($n\in\frac{1}{2}\bbZ$). Let $\clA^\natural_{\leq 1}$ be a $\clA^\natural_0$-submodule of $\clA^\natural$ topologically generated by $\clA^\natural_0$ and $a_n$ ($n\in\frac{1}{2}\bbZ$). 
    \end{dfn}

    By the above observation, we have an exact sequence of Lie superalgebras:
    \[
        \begin{tikzcd}
            0 \ar[r] &\Fun(\clL^{N=1}U) \ar[r] &\clA^\natural_{\leq 1} \ar[r] &\Vect(\clL^{N=1}U) \ar[r] &0 
        \end{tikzcd}
    \]

    Let $\clA^{\natural,N,M}$ be a quotient space of $\clA^\natural$ by a subspace $\sum_{n>N}\clA^\natural a_n+\sum_{m>M}a^*_m\clA^\natural$. Denote by $\clA^{\natural,N,M}_0$, $\clA^{\natural,N,M}_{\leq 1}$ the corresponding subspaces in $\clA^{\natural,N,M}$. $\Vect^{N,M}(\clL^{N=1}U)$ denotes $\clA^{\natural,N,M}_{\leq 1}/\clA^{\natural,N,M}_0$. Similarly, let $\wt{\clA}^{N,M}$ be a quotient space of $\wt{\clA}$ by a left ideal topologically generated by $a_n$ for $n>N$ and $a^*_m$ for $m>M$. Denote by $\wt{\clA}^{N,M}_0$, $\wt{\clA}^{N,M}_{\leq 1}$ the corresponding subspaces in $\wt{\clA}^{N,M}$. 

    \begin{proof}[Proof of \cref{prp:geom interpretation of completed Weyl alg}]
        (1) and (2) follow directly from the definitions. We prove (3).  

         We can define a map $\wt{\clA}^{N,M}_{\leq 1}\to\clA^{\natural,N,M}$ by
        \begin{align*}
            \sum_{0\leq n\leq N}P_n(a^*)a_n&\lmto\sum_{0\leq n\leq N}P_n(a^*)a_n,\\
            \sum_{0\leq m\leq M}a^*_mQ_m(a^*,a) &\lmto \sum_{0\leq m\leq M}Q_m(a^*,a)a^*_m+\sum_{0\leq m\leq M}[a^*_m,Q_m(a^*,a)] 
        \end{align*}
        which is a linear isomorphism, but it is not compatible with the projective systems. However, the composition of this map with the quotient map $\clA^{\natural,N,M}\to\Vect^{N,M}(\clL^{N=1}U)$ is compatible with the projective systems. Thus, the exact sequence
        \[  
            \begin{tikzcd}
            0 \ar[r] &\wt{\clA}^{N,M}_0 \ar[r] &\wt{\clA}^{N,M}_{\leq 1} \ar[r] &\Vect^{N,M}(\clL^{N=1}U) \ar[r] &0 
            \end{tikzcd}
        \]
        induces the desired exact sequence.
    \end{proof}

\bibliographystyle{mybstwithlabels}
\bibliography{references}

\end{document}